\documentclass[12pt]{iopart}

\usepackage{iopams}  
\usepackage{graphicx}
\usepackage[breaklinks=true,colorlinks=true,linkcolor=blue,urlcolor=blue,citecolor=blue]{hyperref}

\usepackage{blindtext}
\usepackage{graphicx}
\usepackage{amsmath}
\usepackage{amssymb}
\usepackage{amsthm}
\usepackage{amsfonts}
\usepackage{float}
\usepackage{caption}
\usepackage{color,linegoal}
\usepackage{multicol}
\usepackage[shortlabels]{enumitem}
\usepackage[bb=boondox,bbscaled=.9,cal=boondox]{mathalfa}

\newtheorem{theorem}{Theorem}
\theoremstyle{definition}
\newtheorem{definition}{Definition}
\newtheorem{lemma}{Lemma}
\newtheorem{corollary}{Corollary}

\numberwithin{equation}{section}
\numberwithin{theorem}{section}
\numberwithin{corollary}{section}
\numberwithin{definition}{section}
\numberwithin{lemma}{section}
\numberwithin{table}{section}
\numberwithin{figure}{section}

\DeclareMathOperator*{\argmin}{arg\,min}

\begin{document}

\begin{center}
\title[]{Estimation of the Distribution of Random Parameters in Discrete Time Abstract Parabolic Systems with Unbounded Input and Output: Approximation and Convergence$^\dagger$}
\end{center}

\author{Melike Sirlanci}
\address{Department of Computing and Mathematical Sciences, California Institute of Technology}
\ead{melikesirlanci@gmail.com}

\author{Susan E. Luczak}
\address{Department of Psychology, University of Southern California}
\ead{luczak@usc.edu}

\author{I. G. Rosen$^*$}
\address{Department of Mathematics, University of Southern California}
\ead{grosen@math.usc.edu}

\begin{abstract}
A finite dimensional abstract approximation and convergence theory is developed for estimation of the distribution of random parameters in infinite dimensional discrete time linear systems with dynamics described by regularly dissipative operators and involving, in general, unbounded input and output operators.  By taking expectations, the system is re-cast as an equivalent abstract parabolic system in a Gelfand triple of Bochner spaces wherein the random parameters become new space-like variables. Estimating their distribution is now analogous to estimating a spatially varying coefficient in a standard deterministic parabolic system. The estimation problems are approximated by a sequence of finite dimensional problems. Convergence is established using a state space-varying version of the Trotter-Kato semigroup approximation theorem.  Numerical results for a number of examples involving the estimation of exponential families of densities for random parameters in a diffusion equation with boundary input and output are presented and discussed.
\end{abstract}

\vspace{2pc}
\noindent{\it Keywords}: Distribution estimation, Random parameters, Distributed parameter systems, Abstract parabolic systems, Regularly dissipative operators.

\vspace{1pc}
\noindent{$\dagger$ This research was supported in part by grants R21AA17711 and R01AA026368-01 from the National Institute of Alcohol Abuse and Alcoholism (NIAAA).  

\noindent$^*$Corresponding author.} 

\section{Introduction}\label{intro}

The work we report on here was motivated by a compound inverse or blind deconvolution problem involving the interpretation of data from a transdermal alcohol biosensor. The observation (dating back to the 1930’s \cite{Palmlov1936,Swift1993,Swift2000,Swift2003,Swift1992}) that ethanol is highly miscible and finds its way into all the water in the body, and in particular, sweat, has in the past two decades, led to the development of technology to measure the amount of ethanol excreted from the body transdermally (i.e. through the skin) through perspiration and to then use it to quantitatively assess intoxication level. The basis for the measurement is an oxidation-reduction (redox) reaction that produces four electrons for each ethanol molecule oxidized.  This results in a continuous current whose level is proportional to the amount of ethanol evaporating from the surface of the skin beneath the sensor.  Now while these devices have been available and in use, both experimentally and commercially, for a number of years, they have been used primarily as abstinence monitors because transdermal alcohol level or concentration (TAC) data cannot consistently be converted to breath and blood alcohol concentrations (BrAC/BAC) across individuals, devices, and environmental conditions.  (BAC and BrAC are currently, and historically have been, the standard measures of intoxication among alcohol researchers and clinicians, as well as in the courts.)  Indeed, unlike a breath analyzer, which relies on a relatively simple model from basic chemistry (i.e., Henry’s Law) for the exchange of gases between circulating pulmonary blood and alveolar air (see, for example, \cite{Lab1990}) that has been found to be reasonably robust across the population, the transport and filtering of alcohol by the skin is physiologically more complex and is affected by a number of factors that differ across individuals (e.g., skin layer thickness, porosity and tortuosity, etc.) and even drinking episodes within individuals (e.g., body and ambient temperature, skin hydration, vasodilation). The challenge in making these devices practicable is to develop a means to reliably convert biosensor measured TAC into BAC or BrAC.  

In our earlier work (\cite{DLR,Dum2008,RLW}) we have taken a strictly deterministic approach to converting TAC to either BAC or BrAC.  We fit first principles physics-based models in the form of a distributed parameter (diffusion) system with unbounded input and output, and used individual calibration data to capture the dynamics of the forward process - the propagation of alcohol from the blood, through the skin, and its measurement by the sensor (i.e. the forward model) – by estimating the parameters (diffusivity, input/output gain, propagation inertia, etc.) that appear in the model via nonlinear least squares.  Then in a second phase of processing, we use the fit model to deconvolve BAC or BrAC from the TAC signal measured by the biosensor in the field.  However, because of the challenges described above, this approach was not entirely satisfying.   Indeed, while it was possible to fit the models quite well to any particular drinking episode, we observed significant variance in the values of the parameters across different individuals and across different drinking episodes for the same individual.   Consequently, the fit models did not yield the desired level of accuracy when they were used to deconvolve BAC or BrAC from TAC for a drinking episode that they were not specifically trained on.

To deal with this problem we have been looking at the idea of fitting a population forward model (having BAC or BrAC as input and TAC as output) in the form of a random partial differential equation, to data from multiple drinking episodes and multiple individuals and then using the population model to solve the deconvolution problem.  Fitting a population model of this form implies that rather than estimate particular values for the parameters, we treat the parameters as random variables and estimate their distributions.   In this way, it will become possible to produce not only an estimate for the BAC or BrAC, but also some form of credible bands to go along with it providing a quantitative estimate of the level of uncertainty in the estimate.

The basic underlying assumption in such an approach is that our first principles physics/physiological based model in essence, describes the dynamics common to the entire population (population interpreted broadly here to include not only all individuals, but also all devices, environmental conditions, and in effect, all ethanol molecules) and to then attribute all unmodeled sources of uncertainty (primarily due to variations in physiology, hardware, and the environment) observed in individual data to random effects. Moreover, we assume that what we observe in any individual data set is the combination or average of these random effects.  Thus, this approach is realized by letting the parameters in the PDE model be random variables, the distributions of which are to be estimated based on aggregate population data. 

In this paper, we develop an abstract approximation framework and convergence theory for formulating and solving just such an estimation problem.  In addition to the theory, we have also included a number of examples and numerical results.  However, we do not discuss here the application of these ideas to either the alcohol biosensor problem discussed above or even the deconvolution problem.  Those results are presented elsewhere (\cite{SRACC,SR2017,SR2018}).  In our treatment here, we are strictly concerned with the problem of estimating the distributions of random parameters in a forward model from a particular class of abstract linear infinite dimensional systems for which the input is known and observations of the output for a sampling of members of the target population are available. That is, we are referring to the problem of fitting the population model.

The class of systems we consider here are those governed by abstract parabolic or hyperbolic operators with damping formulated in a Gelfand triple setting together with input and observations on the boundary of the domain. These types of operators are sometimes referred to as being regularly dissipative, and can typically be shown to generate holomorphic or analytic semigroups.  We formulate the estimation problem in much the same way as it is in standard linear regression. That is, that each data point is assumed to be an observation of the mean population behavior plus random error. We then formulate the estimation problem as an optimization problem over the space of feasible distributions for the random parameters.  The objective of the optimization problem is to minimize prediction error in the form of the difference between the observed output signal and the expectation of the output of the model. We then consider a sequence of approximating estimation problems in each of which the infinite dimensional system is replaced by a finite dimensional approximating system.  We then demonstrate that under appropriate (and readily verifiable) assumptions, the solutions to the approximating estimation problems converge to a solution to the original estimation problem with the infinite dimensional state.  These convergence results are formulated in a functional analytic or operator theoretic setting and are based on ideas and results from linear semigroup theory.  

Our general approach relies heavily on three relatively recent papers: 1) Banks and Thompson's \cite{BT} framework for the estimation of probability measures in random abstract evolution equations and the convergence of finite dimensional approximations in the Prohorov metric, 2) a more recent and enhanced version of the previous paper, \cite{BT1}, and 3) Gittelson, Andreev, and Schwab's \cite{GAS} theory for random abstract parabolic partial differential equations with dynamics defined in terms of coercive sesquilinear forms.  While our effort here is similar in spirit and takes its cue from the treatment in \cite{BT1} and \cite{BT}, it is somewhat different in that we are forced to assume that the probability measures that describe the distribution of our random parameters can be defined in terms of a joint density function; that is, that the random parameters are jointly absolutely continuous.  

The approach in \cite{GAS} is novel in the way that it treats the random parameters in the PDE as another “space-like” independent variable. This is done by appropriately defining corresponding Bochner spaces in which the weak formulation of the problem is stated and shown to be well-posed. In fact, it turns out that the random parameter dependent regularly dissipative operators that determine the underlying PDE are regularly dissipative when embedded in these Bochner spaces. Consequently, we are able to use linear semigroup theory to develop our approximation framework in much the same way as we have in our earlier deterministic treatments. In this way, finite dimensional approximation is handled in much the same way that it is for the standard deterministic space variables, and the estimation of the distribution of the random parameters effectively becomes analogous to the problem of estimating a variable coefficient in a deterministic PDE, a problem which has been studied extensively over the last thirty years (\cite{BL1985} and \cite{BD1985}).

We use the framework in \cite{GAS} together with generation and approximation results from linear semigroup theory, (i.e. the Hille-Yosida-Phillips theorem and a version of the Trotter Kato approximation theorem) to establish that the sufficient conditions for a Banks – Thompson-like convergence result are satisfied.  These theoretical results allow us to develop rigorously established convergent computational algorithms that yield numerical approximations to the desired distributions.  Moreover, the solutions in the Bochner spaces and their finite dimensional approximations directly capture the explicit dependence of the state and output (and eventually the deconvolved input) on the random parameters. Using this together with the estimated distributions for the random parameters, it becomes straight forward to directly identify credible intervals for the output without having to re-solve the PDE many times as you would if you were attempting to identify these credible intervals by naive sampling. 

An outline of the remainder of the paper is as follows. In Section \eqref{est.rand.sys} we formally develop the estimation problem, reformulate it as a nonlinear least squares optimization problem and establish the existence of solutions. In Section \eqref{abs.par} we discuss infinite dimensional systems described by regularly dissipative operators involving unbounded input and output (this is typically the case for a PDE with input and output on the boundary). In Section \eqref{rand.reg} we discuss the framework in \cite{GAS} for treating systems of the form discussed in Section \eqref{abs.par} but now involving random parameters. Our approximation and convergence results are presented in Section \eqref{appr.conv} and a discussion of examples and our numerical results are in Section \eqref{num.res}. Section \eqref{conc.rem} has a few concluding remarks regarding where we plan to go next with this line of research. 

In our discussions to follow we will on occasion use the notation $E[X||f]$, $\mathbb{E}[X||F]$, or $E[X||\pi]$ to denote the expectation of the random variable $X$ with respect to the probability density function $f$, the cumulative distribution function $F$, or the probability measure $\pi$.  We use the "double bar" as opposed to a "single bar" to distinguish what we mean here with conditional expectation.

\section{Estimation of Random Discrete Time Dynamical Systems}\label{est.rand.sys}

We consider the family of discrete or sampled time initial value problems that are set in an, in general, infinite dimensional Hilbert state space, $\mathcal{H}$, given by
\begin{align}
x_{j+1,i}&=g(t_j,x_{j,i},u_i;q), \ \ j=0,...,n_i, \ i=1,2,...,m,\label{eq2.1}\\
x_{0,i}&=x_{0,i}(q), \ \ i=1,2,...,m,\label{eq2.2}
\end{align}
where $g:\mathbb{R}^+ \times \mathcal{H} \times \prod_{j=0}^{n_i}\mathbb{R}^\mu \times Q \rightarrow \mathcal{H}$ and for $j=0,...,n_i$ and $i=1,2,...,m$, $u_i=\{u_{i,j}\}$ is an external input or control with $u_{i,j} \in \mathbb{R}^\mu$, and $t_j=j\tau$, with $\tau>0$ the length of the sampling interval, describing the dynamics of a process common to the entire population. In addition, we assume that we can observe some function of the solutions of \eqref{eq2.1}-\eqref{eq2.2}, $x_{j,i}$, as given by the output equation
\begin{align}
y_{j,i}=y(t_j,x_{0,i},u_i;q)=C(x_{j,i},x_{0,i},u_i;q),\ \ j=0,...,n_i, \ i=1,2,...,m,\label{eq2.3}
\end{align}

\noindent where $C:\mathcal{H} \times \mathcal{H} \times \prod_{j=0}^{n_i}\mathbb{R}^\mu \times Q \rightarrow \mathbb{R}^\nu$. 

In equations \eqref{eq2.1}-\eqref{eq2.3}, we assume $q\in Q$, where $Q$ is the set of admissible parameters (a subset of Euclidean space endowed with Lebesgue measure), and the values of the parameters are specific to each individual in the population. Therefore, assuming that the parameters, $q$, are samples from a random vector $\mathcal{q}$, the objective is to estimate their (joint) distribution based on the aggregate data sampled from the population. For this purpose, we assume that the distribution of these random vectors is described by the joint pdf $f_0\in\mathcal{F}(Q)$, where $\mathcal{F}(Q)$ represents a set of feasible pdfs with support in $Q$. 

There are a number of ways to formulate the statistical model that will be used as the basis for the estimation of the distribution of the random parameters. One approach is to treat \eqref{eq2.1}- \eqref{eq2.3} as an, in general, nonlinear mixed effects model (see, for example, \cite{Davidian,DandG,Demidenko,SR2017}) wherein randomness in the parameters, $q$, are used to quantify uncertainty between subjects, and randomness in the output or measurements, $y_{j,i}$ given in \eqref{eq2.3} is intended to capture uncertainty within individual subjects.  In this case we assume that the observed data points are of the form
\begin{align}
V_{j,i}=y_{j,i} + \varepsilon_{j,i}, \ \ j=0,...,n_i, \ i=1,...,m,\nonumber
\end{align}
where $\varepsilon_{j,i}, \ j=0,...,n_i, \ i=1,...,m$, representing measurement noise are assumed to be independent across subjects (i.e with repsect to $i$), conditionally independent with respect to $\mathcal{q}$ within subjects (i.e with repsect to $j$), identically distributed with mean 0 and known common variance $\sigma^2$, and with $\varepsilon_{j,i} \sim \varphi, \ j=0,...,n_i, \ i=1,...,m$. In this case, for example, using conditional probability and the total probability formula, a likelihood function could be defined formally as
\begin{align}
\begin{aligned}
    \mathcal{L}(f_0;\{V_{j,i}\}) = \prod_{i=1}^m \int_Q L_i(q;\{V_{j,i}\})f_0(q) dq = \prod_{i=1}^m \int_Q \prod_{j=0}^{n_i} \varphi(V_{j,i} - C(x_{j,i},x_{0,i},u_i;q))f_0(q) dq.\nonumber
\end{aligned}
\end{align}
Once one deals with a number of computational issues, specifically, the discretization or parameterization of $f_0$, finite dimensional approximation of the in general infinite dimensional state equation \eqref{eq2.1}, the efficient evaluation of a potentially high dimensional integral, the loss of precision and underflow issues due to the fact that the evaluation of $\mathcal{L}$ requires the computation of products of small numbers, etc., one could then seek a maximum likelihood estimator for $f_0$ by maximizing $\mathcal{L}$ or, more typically, an expression involving $log \mathcal{L}(f_0;\{V_{j,i}\})$ to avoid having to deal with the products. Under appropriate regularity assumptions on $\varphi$, $f_0$, and the system \eqref{eq2.1}- \eqref{eq2.3}, one way to do this might be via a gradient based search.  Another might be via stochastic optimization. One could also treat direct observations of $\mathcal{q}$ as missing data and then use the iterative E-M algorithm to find the MLE (see, for example, \cite{CB}). 

Alternatively, one could use the likelihood function defined above and take a Bayesian approach (see, for example, \cite{Bui,Bui1,IPBF,DS,SR2018,AS}).  One way of doing this would be to assume $f_0=f_0(\cdot;\rho)$ has been parameterized by a parameter vector $\rho \in \mathcal{R}$, where $\mathcal{R}$ denotes a parameter set. Then assume a prior $\mathcal{p}$ on $\rho$ and apply Bayes to obtain the posterior $\hat{\mathcal{p}}$ as
\begin{align}
\begin{aligned}
    \hat{\mathcal{p}}(\rho) = \hat{\mathcal{p}}(\rho|\{V_{j,i}\})         =\frac{1}{Z}\hat{\mathcal{L}}(\rho;\{V_{j,i}\})\mathcal{p}(\rho)=\frac{1}{Z}\mathcal{L}(f_0(\cdot;\rho);\{V_{j,i}\})\mathcal{p}(\rho).\nonumber
\end{aligned}
\end{align}
where $Z$ is the normalizing constant given by
\begin{align}
\begin{aligned}
    Z = \int_\mathcal{R}\hat{\mathcal{L}}(\rho;\{V_{j,i}\})\mathcal{p}(\rho)d\rho = \int_\mathcal{R}\prod_{i=1}^m \int_Q \prod_{j=0}^{n_i} \varphi(V_{j,i} - C(x_{j,i},x_{0,i},u_i;q))f_0(q;\rho)dq \mathcal{p}(\rho)d\rho.\nonumber
\end{aligned}
\end{align}
Still another Bayesian approach could be used to estimate the distribution of $\mathcal{q} \sim f_0$ directly where now the posterior for $\mathcal{q}$, $\hat{\mathcal{p}}=\hat{\mathcal{p}}(q)$ serves as the estimator for $f_0$.  In this case we assume that $\varepsilon_{j,i}, \ j=0,...,n_i, \ i=1,...,m$ are simply independent both across and within subjects, identically distributed with mean 0 and known common variance $\sigma^2$, and with $\varepsilon_{j,i} \sim \varphi, \ j=0,...,n_i, \ i=1,...,m$. If we now let $\mathcal{p}$ denote the prior for $\mathcal{q}$, then Bayes yields
\begin{align}
\begin{aligned}
    \hat{\mathcal{p}}(q) = \hat{\mathcal{p}}(q|\{V_{j,i}\}) =\frac{1}{Z}\prod_{i=1}^m L_i(q;\{V_{j,i}\})\mathcal{p}(q) = \frac{1}{Z}\prod_{i=1}^m \prod_{j=0}^{n_i} \varphi(V_{j,i} - C(x_{j,i},x_{0,i},u_i;q))\mathcal{p}(q),\nonumber
\end{aligned}
\end{align}
where the normalizing constant $Z$ is now given by
\begin{align}
\begin{aligned}
    Z = \int_Q \prod_{i=1}^m L_i(q;\{V_{j,i}\})\mathcal{p}(q)dq = \int_Q \prod_{i=1}^m\prod_{j=0}^{n_i} \varphi(V_{j,i} - C(x_{j,i},x_{0,i},u_i;q))\mathcal{p}(q)dq.\nonumber
\end{aligned}
\end{align}
Both of these Bayesian approaches also have some of the same computational issues as the MLE approach when some sort of MCMC technique such as Metropolis-Hastings or the Gibbs Sampler is used to sample the posterior distribution. 

In our study here, however, we take a statistically somewhat less sophisticated approach.  We consider the naive pooled data estimator.  We do this for a number of reasons. 1) Our primary focus here is the finite dimensional approximation of the infinite dimensional state equation and the convergence of the corresponding estimators and the computational challenges described above would only serve to confound our findings, 2) The naive pooled estimator meshes especially well with the approach we take in dealing with the randomness in the family of PDEs (i.e. abstract parabolic, and eventually, damped hyperbolic) of particular interest to us here in the context of the alcohol biosensor problem described earlier. 3) A reasonable argument could be made that the data we observe is best described as pooled or averaged.  We note that it in fact turns out that the approximation and convergence results we present here are highly relevant to the MLE and Bayesian approaches described in the previous paragraphs; we are currently investigating that and we will report on our findings and results in those cases elsewhere.  Finally it is interesting to note that in the Bayesian approach, if the prior $f_0$ and the distribution of the measurement noise process, $\varepsilon_{j,i}$, as described by the density $\varphi$ are both assumed to be normal, then the naive pooled data estimator we find here is in fact the Maximum A-Posteriori, or MAP,  estimator.     

In light of this, our statistical model assumes that the observed data points can be represented by the mean output of the model plus random error. Thus, we assume that we have random observations of the process given by a random array with components
\begin{align}
V_{j,i}=\mathbb{E}[y_{j,i}||f_0]+\varepsilon_{j,i}, \ \ j=0,...,n_i, \ i=1,...,m,\label{eq2.4}
\end{align}
where in \eqref{eq2.4}, $\varepsilon_{j,i}, \ j=0,...,n_i, \ i=1,...,m$, represent measurement noise and are assumed to be independent and identically distributed with mean 0 and known common variance $\sigma^2$. For $f\in\mathcal{F}(Q)$, define
\begin{align}
v_i(t_j;f)=\mathbb{E}[y(t_j,x_{0,i},u_i;\mathcal{q})||f]=\int_Q C(x_{j,i},x_{0,i},u_i;q)f(q)dq,\label{eq2.5}
\end{align}
the mean behavior at time $t_j$, $j=0,...,n_i$, if $\mathcal{q}\sim f$.

The estimation problem is to estimate the pdf, $f_0$, using a least squares approach 
\begin{align}
\hat{f}=\argmin_{f\in\mathcal{F}(Q)}J(f;V)=\argmin_{f\in\mathcal{F}(Q)}\sum_{i=1}^m\sum_{j=0}^{n_{i}}(V_{j,i}-v_i(t_j;f))^2.\label{eq2.6}
\end{align}
where the $v_i(t_j;f)$ are as given in \eqref{eq2.5}.

Solving the optimization problem given in \eqref{eq2.6} will typically require finite dimensional approximation of the dynamical system given in \eqref{eq2.1}-\eqref{eq2.2}, and the parameterization of the feasible set of pdfs, $\mathcal{F}(Q)$. Indeed, in our treatment here, we assume that the set of pdfs, $\mathcal{F}(Q)$, is parameterized by a vector of parameters $\theta\in\Theta$, where $\Theta\subseteq\mathbb{R}^r$ is a set of feasible parameters. In this case, we denote the set of pdfs by $\mathcal{F}_{\Theta}(Q)$. 

We approximate the estimation problem given in \eqref{eq2.6} by a sequence of finite dimensional estimation problems by replacing $v_i(t_j;f)$ with a finite dimensional approximation $v_i^N(t_j;f)$.  We obtain  
\begin{align}
    \hat{f}^N=\argmin_{f\in\mathcal{F}_{\Theta}(Q)}J^N(f;V)=\argmin_{f\in\mathcal{F}_{\Theta}(Q)}\sum_{i=1}^m\sum_{j=0}^{n_i}(V_{j,i}-v_i^N(t_j;f))^2.\label{eq2.7}
\end{align}

We note that ultimately, we will want to dispense with the assumption that $\mathcal{F}(Q)$ has been parametrized by the finite dimensional parameter $\theta\in\Theta$ and actually estimate the shape of $f$ directly. In this case, $\mathcal{F}(Q)$ will also have to be approximated or discretized with the level, or dimension of the parameterization having to grow in order to establish convergence. We are currently studying this extension to the results presented here and will discuss our findings elsewhere. Analogous to theorem 5.1 in \cite{BT}, we have the following convergence result for the $\hat{f}^{N}$'s.

\begin{theorem}\label{th2.1}
Let $\Theta\subseteq\mathbb{R}^r$ be compact. If

\textbf{A.} The maps on $\Theta$, $\theta\mapsto f(q;\theta)$, for almost every $q\in Q$, and $\theta\mapsto J^N(f(\cdot;\theta);V)$, for all $N$ and $f\in\mathcal{F}_{\Theta}(Q)$ are continuous,

\textbf{B.} For any sequence of densities $f_N\in\mathcal{F}_{\Theta}(Q)$ with $\lim_{N\rightarrow\infty}f_N(q)=f(q)$, a.e. $q\in Q$, for some $f\in\mathcal{F}_{\Theta}(Q)$, we have $v_i^N(t_j;f_N)$ converging to $v_i(t_j;f)$ for all $i\in\{1,...,m\}$ and $j\in\{0,...,n_i\}$ as $N\rightarrow\infty$, and

\textbf{C.} The $v_i(t_j;f)$ and $v_i^N(t_j;f)$ are uniformly bounded for all $j\in\{0,...,n_i\}$, $i\in\{1,...,m\}$ and $f\in\mathcal{F}_{\Theta}(Q)$,

\noindent then it will follow that there exist solutions $\hat{f}^{N}$ to the estimation problems over $\mathcal{F}_{\Theta}(Q)$, given in \eqref{eq2.7}, and there exists a subsequence of the $\hat{f}^{N}$'s that converges to a solution $\hat{f}$ of the estimation problem over $\mathcal{F}_{\Theta}(Q)$ given in \eqref{eq2.6}.
\end{theorem}

\begin{proof} Finding the solution to the problem in \eqref{eq2.7} is equivalent to finding the parameters $\theta\in\Theta$ such that $J^N(f;V)$ is minimized. Since $\Theta$ is a compact set and the map $\theta\rightarrow J^N(f(\cdot;\theta);V)$ is continuous for all $N$ by (A), a solution $\hat{f}^{N}$ to the estimation problem \eqref{eq2.7} over $\mathcal{F}_{\Theta}(Q)$ exists.

Next, let $\{f_N\}\subseteq\mathcal{F}_{\Theta}(Q)$ be any sequence with $\lim_{N\rightarrow\infty}f_N(q)=f(q)$, a.e. $q\in Q$ for some $f\in\mathcal{F}_{\Theta}(Q)$ and consider that
\begin{align}
    \begin{aligned}
    |J^N(f_N;V)-J(f;V)|&=|\sum_{i=1}^m\sum_{j=0}^{n_i}(V_{j,i}-v_i^N(t_j;f_N))^2\\
    &\quad-\sum_{i=1}^m\sum_{j=0}^{n_i}(V_{j,i}-v_i(t_j;f))^2|\\
    &\leq\sum_{i=1}^m\sum_{j=0}^{n_i}|2V_{j,i}-(v_i(t_j;f)+v_i^N(t_j;f_N))|\\
    &\quad\cdot|v_i(t_j;f)-v_i^N(t_j;f_N)|\\
    &\leq M\sum_{i=1}^m\sum_{j=0}^{n_i}|v_i(t_j;f)-v_i^N(t_j;f_N)|,\nonumber
    \end{aligned}
\end{align}
for some $M>0$, since $v_i(t_j;f)$ and $v_i^N(t_j;f)$ are uniformly bounded for all $i\in\{1,...,m\}$ and $j\in\{0,...,n_i\}$ (by assumption (C)), and $f\in\mathcal{F}_{\Theta}(Q)$. Then, by (B), we obtain
\begin{align}
J^N(f_N;V)\rightarrow J(f;V),\label{eq2.8}
\end{align}
as $N\rightarrow\infty$. On the other hand, since $\hat{f}^{N}=\hat{f}(\cdot;\hat{\theta}^N)$, where $\hat{\theta}^N\in\Theta$, is the minimizer of $J^N(f;V)$, we have
\begin{align}
J^N(\hat{f}^{N};V)\leq J^{N}(f;V),\label{eq2.9}
\end{align}
for all $f=f(\cdot;\theta)\in\mathcal{F}_{\Theta}(Q)$ and $N = 1,2,...$. Since $\{\hat{\theta}^N\}\subset\Theta$, compact, there exists a subsequence $\hat{\theta}^{N_k}$ with $\hat{\theta}^{N_k}\rightarrow\hat{\theta}$ as $k\rightarrow\infty$. Thus, taking the limit as $k\rightarrow\infty$ in \eqref{eq2.9} with $N$ replaced by $N_k$, and using \eqref{eq2.8} (with $f^N_k = f$, all $k = 1,2,...$ when the limit is taken on the right hand side of \eqref{eq2.9}), we obtain
\begin{align}
J(\hat{f};V)\leq J(f;V), \label{eq2.10}
\end{align}
for all $f\in\mathcal{F}_{\Theta}(Q)$, where $\hat{f} = \hat{f}(\cdot;\hat{\theta})$. Thus, \eqref{eq2.10} implies that $\hat{f}$ is a solution of estimation problem given in \eqref{eq2.6} over $\mathcal{F}_{\Theta}(Q)$.
\end{proof}

\section{Abstract Parabolic Systems with Unbounded Input and Output}\label{abs.par}

Let $V$ and $H$ be in general complex (but in many instances, real would suffice) Hilbert spaces with $V\hookrightarrow H$, i.e. $V$ is continuously and densely embedded in $H$. By identifying $H$ with its dual $H^*$, we obtain the Gelfand triple $V\hookrightarrow H\hookrightarrow V^*$. Let $<\cdot,\cdot>_H$ denote the $H$ inner product and $|\cdot|_H$, $||\cdot||_V$ denote norms on $H$ and $V$, respectively, and assume that $(Q,d_Q)$ is a compact metric space contained in Euclidean space endowed with Lebesgue measure. In what follows all multi-dimensional vectors, whether in Euclidean or some abstract space, are assumed to be column vectors, unless explicitly stated otherwise.  For $q\in Q$, let $a(q;\cdot,\cdot):V\times V\rightarrow\mathbb{C}$ be a sesquilinear form that has the following properties

\begin{enumerate}[i.]
\item \textbf{Boundedness} \ \ There exists a constant $\alpha_0 > 0$ such that $|a(q;\psi_1,\psi_2)|\leq\alpha_0||\psi_1||_V ||\psi_2||_V$, \ \ $\psi_1,\psi_2\in V$, \ $q\in Q$,
\item \textbf{Coercivity} \ \ There exist constants $\lambda_0 \in \mathbb{R}$ and $\mu_0 > 0$ such that $a(q_1;\psi,\psi)+\lambda_0|\psi|_H^2\geq\mu_0||\psi||_V^2$, \ \ $\psi\in V$, \ $q\in Q$,
\item \textbf{Measurability} \ \ For all $\psi_1,\psi_2\in V$, the map $q\mapsto a(q;\psi_1,\psi_2)$ is measurable on\ \ $Q$ with respect to all measures defined in terms of the densities in $\mathcal{F}_{\Theta}(Q)$, where $\Theta\subseteq\mathbb{R}^r$ is the set of feasible parameters.
\end{enumerate}

Assume further that $b(q),c(q)$ are respectively $\mu$ and $\nu$ dimensional row vectors in $V^*$ with the maps $q\mapsto<b(q),\psi>_{V^*,V}$ and $q\mapsto<c(q),\psi>_{V^*,V}$ measurable on $Q$ for $\psi \in V$, where $<\cdot,\cdot>_{V^*,V}$ denotes the duality pairing between $V$ and $V^*$.  We consider the system which is written in weak form as
\begin{align}
    \begin{aligned}
        \left<\dot{x},\psi\right>_{V^*,V}+a(q;x,\psi)&=\left<b(q),\psi\right>_{V^*,V}u, \ \ \psi\in V,\\
        x(0)&=x_{0} \in H,\\
        y(t)&=\int_{0}^{T}\left<c(q),x_{(t)}(s)\right>_{V^*,V}ds,\label{eq3.1}
    \end{aligned}
\end{align}
where $T>0$, and $\varphi_{(t)}(s)=\varphi(t-s)\chi_{[0,t]}(s)$, $s \in [0,T]$. For $u\in L_2([0,T],\mathbb{R}^\mu)$, it can be shown that \eqref{eq3.1} has a unique solution (see \cite{LJL,TNB}) $x\in W(0,T):=\{\psi:\psi\in L_2([0,T],V), \dot{\psi} \in L_2([0,T],V^*)\} \subseteq C([0,T],H)$ which depends continuously on $u\in L_2([0,T],\mathbb{R}^\mu)$. It follows that $y\in L_2([0,T],\mathbb{R}^\nu)$.

For $q\in Q$, under the assumptions (i),(ii), the sesquilinear form $a(q;\cdot,\cdot)$ defines a bounded linear operator $A(q):V\rightarrow V^*$ by $<A(q)\psi_1,\psi_2>_{V^*,V}=-a(q;\psi_1,\psi_2)$ where $\psi_1,\psi_2\in V$. It can be shown further that (see \cite{BI,BK,TNB}) $A(q)$ restricted to the set $Dom(A(q)) = \{\phi \in V: A(q)\phi \in H\}$ is the infinitesimal generator of a holomorphic or analytic semigroup of bounded linear operators on $H$. Moreover, this semigroup can be restricted to be a holomorphic semigroup on $V$ and extended to a holomorphic semigroup on $V^*$ by appropriately restricting or extending the domain, $Dom(A(q))$, of the operator $A(q)$ (see, for example, \cite{BI} and \cite{TNB}).

For $q\in Q$, define the operators $B(q):\mathbb{R}^{\mu}\rightarrow V^*$ by $\left<B(q)u,\varphi\right>_{V^*,V}=\left<b(q),\varphi\right>_{V^*,V}u$ and $C(q):L_2([0,T],V)\rightarrow\mathbb{R}^{\nu}$ by $C(q)\psi=\int_{0}^{T}\left<c(q),\psi(s)\right>_{V^*,V}ds$, for $u\in\mathbb{R}^{\mu}$, $\varphi\in V$, and $\psi \in L_2([0,T],V)$, and rewrite the system in \eqref{eq3.1} as
\begin{align}
\begin{aligned}
    \dot{x}(t)&=A(q)x(t)+B(q)u(t),\\
    x(0)&=x_0,\\
    y(t)&=C(q)x_{(t)}, \ \ t>0.\label{eq3.2}
\end{aligned}
\end{align}
The mild solution of \eqref{eq3.2} is given by the variation of constants formula as
\begin{align}
    x(t;q)=e^{A(q)t}x_0+\int_0^t e^{A(q)(t-s)}B(q)u(s)ds, \ \ t\geq 0.\label{eq3.3}
\end{align}
Moreover, since the semigroup $\{e^{A(q)t}: t \geq 0 \}$ is analytic it follows that 
\begin{align}
    y(t;q)=C(q)x_{(t)}(q) = \int_{0}^{T}\left<c(q),x_{(t)}(s;q)\right>_{V^*,V}ds, \ \ t\geq 0.\label{eq3.4}
\end{align}
is well defined.

\subsection{The Discrete Time Formulation}\label{discrete.time}

Now let $\tau > 0$ be a sampling time and consider zero-order hold inputs of the form $u(t)=u_j$, $t\in[j\tau,(j+1)\tau)$, $j=0,1,2,...$. Setting $x_j=x(j\tau)$,  for $j=0,1,2,...$, \eqref{eq3.3} and \eqref{eq3.4} yield that
\begin{align}
    x_{j+1}=\hat{A}(q)x_j+\hat{B}(q)u_j, \ \ y_j=\hat{C}(q)x_{(j)}, \ \ j=0,1,2,...\label{eq3.5}
\end{align}
where now we let $x_0\in V$. Here, again by the properties of the analytic semigroup (see \cite{pazy83,TNB}), we have $\{e^{A(q)t}:t\geq0\}$, $x_j \in V$, $\hat{A}(q)=e^{A(q)\tau}\in\mathcal{L}(V,V)$ and $\hat{B}(q)=\int_0^{\tau}e^{A(q)s}B(q)ds\in\mathcal{L}(\mathbb{R}^{\mu},H)$. The operator $\hat{C}(q)$ appearing in \eqref{eq3.5} is defined by recalling \eqref{eq3.4}. We set 
\begin{align}
\hat{C}(q)x_{(j)} = C(q)x_{(j)},\label{eq3.6}
\end{align}
where $x_{(j)}$ in \eqref{eq3.6} denotes the function in $L_2(0,T,V)$ given by 
\begin{align}
x_{(j)}=\sum_{i=1}^{j}x_{i}\chi_{[(j-1)\tau,j\tau)}.\label{eq3.6a}
\end{align}

Now, in light of the coercivity assumption, Assumption (ii), by making the change of variables $z(t) = e^{-\lambda_0t}x(t)$ and $v(t) = e^{-\lambda_0t}u(t)$, without loss of generality we may assume that the operator $A(q)$ is invertible with bounded inverse. Thus we have that $\hat{B}(q)=\int_0^{\tau}e^{A(q)s}B(q)ds=A(q)^{-1}e^{A(q)s}B(q)\Big|_0^{\tau}=(\hat{A}(q)-I)A(q)^{-1}B(q) \in {\mathcal{L}}(\mathbb{R}^{\mu},V)$. It follows that the recurrence given in \eqref{eq3.5} is a recurrence in $V$ with $\hat{A}(q) \in {\mathcal{L}}(V,V)$ and $\hat{B}(q) \in {\mathcal{L}}(\mathbb{R}^{\mu},V)$. Thus it now becomes possible to allow the discrete time output operator $\hat{C}(q) \in {\mathcal{L}}(V,\mathbb{R}^{\nu})$ defined in \eqref{eq3.6} and \eqref{eq3.6a}, if so desired, to take on the much simpler form $\hat{C}(q)x = \left<c(q),x\right>_{V^*,V}$.  In what follows we shall assume that the output operator takes this simpler form.

\subsection{Systems with Boundary Input}\label{boundary.input}

Of primary interest to us here are systems of the form \eqref{eq3.1} or \eqref{eq3.2} where the input $u$ is on the boundary of the spatial domain.  The theory developed in \cite{CS} and \cite{PS} tells us how in this case to define the input operator $B(q)$ and the notion of a mild solution upon which our approach is based. Let $W$ be a Hilbert space which is densely and continuously embedded in $H$. Let $\Delta (q) \in {\mathcal{L}}(W,H)$ and $\Gamma (q) \in {\mathcal{L}}(W,\mathbb{R}^{\mu})$ and assume that $Dom(A(q)) \subseteq {\mathcal N}(\Gamma(q)) \subseteq W$, $\Gamma (q)$ is surjective and $\Delta (q) = A(q)$ on $Dom(A(q))$. We then consider the system with input on the boundary given by  
\begin{align}
    \begin{aligned}
    \dot{x}(t)&=\Delta (q)x(t), \ \ t>0,\\
    \Gamma(q)x(t)&=u(t), \ \ t>0,\\
    y(t)&=C(q)x_{(t)}, \ \ t>0,\\
    x(0)&=x_0.\label{eq3.6b}
    \end{aligned}
\end{align}
In \cite{CS}, Curtain and Salamon define a solution to the system \eqref{eq3.6b} for the case where $u \in C([0,T];\mathbb{R}^\mu)$ and $x_0 \in W$ with $\Gamma(q)x_0=u(0)$, to be a function $x \in C([0,T];W)\cap C^1([0,T];H)$ that satifies \eqref{eq3.6b} at every $t \in (0,T)$.  The operator $A(q)$ densely defined implies that it has an adjoint operator $A(q)^*:Dom(A(q)^*)\subseteq H\rightarrow H$ which is also densely defined and closed. Defining $Z^*$ to be the Hilbert space $Dom(A(q)^*)$ endowed with the graph Hilbert space norm associated with $A(q)^*$, $Z^*$ will be continuously and densely embedded in $H$. So, the Gelfand triple $Z^*\hookrightarrow H\hookrightarrow Z$ is obtained where $Z = Z^{**}$ represents the dual space of $Z^*$. By definition $A(q)^*\in {\mathcal L}(Z^*,H)$ and consequently therefore, $A(q) \in {\mathcal L}(H,Z)$. It follows that the semigroup $\{e^{A(q)t}:t\geq 0\}$ can be uniquely extended to a holomorphic semigroup on $Z$ with infinitesimal generator $A(q):H\subseteq Z\rightarrow Z$, the extension $A(q)$ to $H$ defined via the duality pairing $<A(q)\psi,\phi>_{Z,Z^*}=<\psi,A(q)^*\phi>_H$, for $\psi\in H$, and $\phi \in Z^*=Dom(A(q)^*)$.

For each $q\in Q$, let $\Gamma^+(q)\in\mathcal{L}(\mathbb{R}^{\mu},W)$ be any right inverse of $\Gamma(q)\in\mathcal{L}(W,\mathbb{R}^{\mu})$, and define the operator $B(q)\in\mathcal{L}(\mathbb{R}^{\mu},Z)$ by $B(q)=(\Delta(q)-A(q))\Gamma^+(q)$. It is not difficult to show that $B(q)$ is well defined (i.e. that it does not depend on the particular choice of the right inverse $\Gamma^+(q)$). Then for any $x_0 \in H$ and $u \in L_2([0,T];\mathbb{R}^\mu)$, the mild solution, $x \in C([0,T];Z)$, of the initial boundary value problem in \eqref{eq3.6b} is the $Z$-valued function given by
\begin{align}
        x(t)=e^{A(q)t}x_0+\int_0^t e^{A(q)(t-s)}B(q)u(s)ds, \ \ t\geq 0.\label{eq3.7}
\end{align}
It is shown in \cite{CS} that if \eqref{eq3.6b} has a solution, then it is given by \eqref{eq3.7} where  $x\in C([0,T],H)\cap H^1((0,T),Z)$ and moreover, we have that the estimate given by $|\int_0^t e^{A(q)(t-s)}B(q)u(s)ds|_{H} \leq k||u||_{L_2([0,T];\mathbb{R}^\mu)}$ holds.

We note that if in fact we have that $W \subset V$, which is often the case (for example, in a one dimensional diffusion equation with either Neumann or Robin boundary input (see our examples in Section \eqref{num.res} below), but may not be the case if, for example, the boundary input is Dirichlet), then in the above formulation we may take $Z^* = V$ and $Z = V^*$.  In this case it will follow that $B(q)=(\Delta(q)-A(q))\Gamma^+(q) \in \mathcal{L}(\mathbb{R}^{\mu},V^*)$ and consequently that the theory presented at the beginning of Section \eqref{abs.par}, and in particular, the discrete time theory presented in Section \eqref{discrete.time}, applies. For ease of exposition, we will assume that this is indeed the case for what follows below.  We note that all the results continue to follow in the more general case where $Z^* = Dom(A(q)^*)$.  It then follows that $\hat{A}(q)=e^{A(q)\tau} \in \mathcal{L}(V,V)$ and that $\hat{B}(q)=\int_0^{\tau}e^{A(q)s}dsB(q)\in\mathcal{L}(\mathbb{R}^{\mu},V)$ and therefore that
\begin{align}
\hat{B}(q)=\int_0^{\tau}e^{A(q)s}B(q)ds
=A(q)^{-1}e^{A(q)s}B(q)\Big|_0^{\tau}
=(\hat{A}(q)-I)A(q)^{-1}B(q),\nonumber
\end{align}
and $\hat{C}(q)=C(q)\in\mathcal{L}(V,\mathbb{R}^{\nu})$.  Note that now we have
\begin{align}
\hat{B}(q)=(I-\hat{A}(q))\Gamma^+(q)+\int_0^{\tau}e^{A(q)s}ds \Delta(q) \Gamma^+(q)\in\mathcal{L}(\mathbb{R}^{\mu},V),\label{eq3.10}
\end{align}
and if $\Gamma^+(q)$ can be chosen so that $R(\Gamma^+(q))\in \mathcal{N}(\Delta(q))$, then the expression in \eqref{eq3.10} becomes $\hat{B}(q)=(I-\hat{A}(q))\Gamma^+(q)$. Then, if $x_0=0\in H$, $y_i$ is given by
\begin{align}
    \begin{aligned}
        y_i&=\sum_{j=0}^{i-1}C(q)\hat{A}(q)^{i-j-1}\hat{B}(q)u_i\\
        &=\sum_{j=0}^{i-1}K_{i,j}u_j, \ \ i=1,2,... ,\label{eq3.11}
    \end{aligned}
\end{align}
where the operator $K_{i,j}=C(q)\hat{A}(q)^{i-j-1}(I-\hat{A}(q))\Gamma^+(q)$ appearing in \eqref{eq3.11} is the gain that represents the contribution of the $j^{th}$ input channel to the $i^{th}$ output channel.

\section{Random Regularly Dissipative Operators and Their Associated Semigroups}\label{rand.reg}

In this section, we summarize the key ideas from the framework developed in \cite{GAS} and \cite{SGJ} which are central to our approach. We assume that $\mathcal{q}$ is a $p$-dimensional random vector whose support is in $\prod_{i=1}^p[a_i,b_i]$ where $-\infty<\bar{\alpha}<a_i<b_i<\bar{\beta}<\infty$ for all $i=1,2,...,p$. Letting $\vec{a}=[a_i]_{i=1}^p$, $\vec{b}=[b_i]_{i=1}^p$ and let $\Theta\subset\mathbb{R}^r$ for some $r$ be closed and bounded. We assume that the distribution of $\mathcal{q}$ can be represented by an absolutely continuous cumulative distribution function $F(q;\vec{a},\vec{b},\vec{\theta})$, or equivalently, by a (push forward) measure  $\pi=\pi(\vec{a},\vec{b},\vec{\theta})$, where $\vec{\theta}\in\Theta$. Let $a(\cdot;\cdot,\cdot)$ be a sesquilinear form satisfying (i)-(iii) given in Section \eqref{abs.par}, where the assumed measurability is with respect to all of the measures $\pi=\pi(\vec{a},\vec{b},\vec{\theta})$. 

Define the Bochner spaces $\mathcal{V} = L_{\pi}^2(Q;V)$ and $\mathcal{H}=L_{\pi}^2(Q;H)$.  The assumptions from Section \eqref{abs.par} on the spaces $V$ and $H$ guarantee that the spaces $\mathcal{V}$, $\mathcal{H}$ and $\mathcal{V}^{*}$ form the Gelfand triple $\mathcal{V}\hookrightarrow\mathcal{H}\hookrightarrow\mathcal{V}^*$ (see \cite{GAS}) where $\mathcal{H}$ is identified with its dual $\mathcal{H}^*$ and $\mathcal{V}^*$ is identified with $L_{\pi}^2(Q;V^*)$.

For $\vec{a}=[a_i]_{i=1}^p$, $\vec{b}=[b_i]_{i=1}^p$ satisfying $-\infty<\bar{\alpha}<a_i<b_i<\bar{\beta}<\infty$ for all $i=1,2,...,p$, and $\vec{\theta}\in\Theta$, set $\rho = (\vec{a},\vec{b},\vec{\theta})$.  Then we define the $\pi(\rho)$-averaged sesquilinear forms $\mathcal{a}(\rho;\cdot,\cdot):\mathcal{V} \times \mathcal{V}\rightarrow\mathbb{C}$ (note, the spaces $\mathcal{H}$, $\mathcal{V}$, and $\mathcal{V^*}$ now of course depend on $\rho$, but our notation here we will not explicitly show this dependence unless clarity demands it) by
\begin{align}
    \mathcal{a}(\rho;\varphi,\psi)=\int_Q a(q;\varphi(q),\psi(q))d\pi(q;\rho)=\mathbb{E}[a(\mathcal{q};\varphi(\mathcal{q}),\psi(\mathcal{q}))||\pi(\rho)],\label{eq4.1}
\end{align}
where $\varphi,\psi\in\mathcal{V}$ and $\rho = (\vec{a},\vec{b},\vec{\theta})$. It is not difficult to show that Assumptions (i)-(iii) imply that $\mathcal{a}(\rho;\cdot,\cdot)$ is a bounded and coercive sesquilinear form on $\mathcal{V}\times\mathcal{V}$. Consequently, this sesquilinear form defines a bounded linear map $\mathcal{A}(\rho):\mathcal{V}\rightarrow\mathcal{V}^*$ by $<\mathcal{A}(\rho)\varphi,\psi>_{\mathcal{V}^*,\mathcal{V}}=-\mathcal{a}(\rho;\varphi,\psi)$ which when appropriately restricted or extended is the infinitesimal generator of analytic semigroups of bounded linear operators $\{e^{\mathcal{A}(\rho)t}:t\geq0\}$ on  $\mathcal{V}$, $\mathcal{H}$ and $\mathcal{V}^*$ (see \cite{BI,BK,TNB}). 
We assume that the maps $q\mapsto<b(q),\psi(q)>_{V^*,V}$ and $q\mapsto<c(q),\psi(q)>_{V^*,V}$ are $\pi(\rho)$-measurable for any $\psi\in\mathcal{V}$, and that $||b(q)||_{V^*}$, $||c(q)||_{V^*}$ are uniformly bounded for a.e. $q\in Q$. We then define $\mathcal{B}(\rho):\mathbb{R^{\mu}}\rightarrow\mathcal{V}^*$ and $\mathcal{C}(\rho):\mathcal{V}\rightarrow\mathbb{R^{\nu}}$ by
\begin{align}
    <\mathcal{B}(\rho)u,\psi>_{\mathcal{V}^*,\mathcal{V}}&=\int_Q\left<b(q),\psi(q)\right>_{V^*,V}d\pi(q;\rho)u\nonumber\\
    &=\mathbb{E}[\left<b(\mathcal{q}),\psi(\mathcal{q})\right>_{V^*,V}||\pi(\rho)]u,\label{eq4.2}
\end{align}
\begin{align}
    \mathcal{C}(\rho)\psi=\int_Q\left<c(q),\psi\right>_{V^*,V}d\pi(q;\rho)=\mathbb{E}[\left<c(\mathcal{q}),\psi(\mathcal{q})\right>_{V^*,V}||\pi(\rho)],\label{eq4.3}
\end{align}
for $u\in\mathbb{R}^{\mu}$ and $\psi\in\mathcal{V}$.

With the definitions \eqref{eq4.1} - \eqref{eq4.3} of the operators $\mathcal{A}$, $\mathcal{B}$, and $\mathcal{C}$, consider the abstract evolution system given by 
\begin{align}
    \begin{aligned}
        \dot{\mathcal{x}}(t)&=\mathcal{A}(\rho)\mathcal{x}(t)+\mathcal{B}(\rho)u(t),\\
        \mathcal{x}(0)&=\mathcal{x}_0 \in H,\\
        \mathcal{y}(t)&=\mathcal{C}(\rho)\mathcal{x}(t), \ t>0,\label{eq4.4}
    \end{aligned}
\end{align}
whose mild solution is given by
\begin{align}
    \mathcal{x}(t)=\mathcal{T}(t;\rho)\mathcal{x}_0+\int_0^t\mathcal{T}(t-s;\rho)\mathcal{B}(\rho)u(s)ds, \ t\geq0,\label{eq4.5}
\end{align}
where $\mathcal{T}(t;\rho)=\{e^{\mathcal{A}(\rho)t}:t\geq0\}$ is the analytic semigroup generated by the operator $\mathcal{A}(\rho)$. From \eqref{eq4.4} and \eqref{eq4.5}, it follows that
\begin{align}
    \mathcal{y}(t)=\int_0^t\mathcal{C}(\rho)\mathcal{T}(t-s;\rho)\mathcal{B}(\rho)u(s)ds, \ t\geq0.\label{eq4.6}
\end{align}
As in Section \eqref{abs.par}, we obtain a discrete or sampled time version of \eqref{eq4.4}. Now let $x_0\in V$, let $\tau>0$ be the sampling time, and consider zero-order hold inputs of the form $u(t)=u_j$, $t\in[j\tau,(j+1)\tau)$, $j=0,1,2,...$. Setting $\mathcal{x}_j=\mathcal{x}(j\tau)$ and $\mathcal{y}_j=\mathcal{y}(j\tau)$, $j=0,1,2,...$, \eqref{eq4.5} and \eqref{eq4.6} yield
\begin{align}
    \mathcal{x}_{j+1}=\hat{\mathcal{A}}(\rho)\mathcal{x}_j+\hat{\mathcal{B}}(\rho)u_j, \ \ \mathcal{y}_j=\hat{\mathcal{C}}(\rho)\mathcal{x}_j, \ \ j=0,1,2,...,\label{eq4.7}
\end{align}
with $\mathcal{x}_0\in\mathcal{V}$ and $\hat{\mathcal{A}}(\rho)=\mathcal{T}(\tau;\rho)\in\mathcal{L}(\mathcal{V},\mathcal{V})$, $\hat{\mathcal{B}}(\rho)=\int_0^{\tau}\mathcal{T}(s;\rho)\mathcal{B}(\rho)ds\in\mathcal{L}(\mathbb{R}^{\mu},\mathcal{V})$, and $\hat{\mathcal{C}}(\rho)=\mathcal{C}(\rho)\in\mathcal{L}(\mathcal{V},\mathbb{R}^{\nu})$. Note that the operators $\hat{\mathcal{A}}(\rho)$ and $\hat{\mathcal{B}}(\rho)$ are bounded since $\{\mathcal{T}(t;\rho):t\geq0\}$ is an analytic semigroup on $\mathcal{V}$, $\mathcal{H}$, and $\mathcal{V}^*$ (see \cite{BI,BK,LJL,TNB}). If $\mathcal{A}(\rho):Dom(\mathcal{A}(\rho))\subseteq\mathcal{V}^*\rightarrow\mathcal{V}^*$ has bounded inverse, then $\hat{\mathcal{B}}(\rho)=\int_0^{\tau}\mathcal{T}(s;\rho)\mathcal{B}(\rho)ds=\hat{\mathcal{A}}(\rho)^{-1}\mathcal{T}(s;\rho)\mathcal{B}(\rho)\big|_0^{\tau}=(\hat{\mathcal{A}}(\rho)-I)\mathcal{A}(\rho)^{-1}\mathcal{B}(\rho)\in\mathcal{L}(\mathbb{R}^\mu,\mathcal{V})$.

It is shown in \cite{GAS} and \cite{SGJ} that the solutions of systems \eqref{eq4.4} and \eqref{eq3.2} and \eqref{eq4.7} and \eqref{eq3.5} agree for $\pi$-a.e. $q\in Q$. It follows that
\begin{align}
    \mathcal{y}(t)=\mathcal{C}(\rho)\mathcal{x}(t)=\mathbb{E}[y(t;\mathcal{q})||\pi(\rho)]=\mathbb{E}[C(\mathcal{q})x(t;\mathcal{q})||\pi(\rho)], \ \ \forall t\geq0,\label{eq4.8}
\end{align}
and hence, from \eqref{eq4.8}, that
\begin{align}
    \mathcal{y}_j=\hat{\mathcal{C}}(\rho)\mathcal{x}_j=\mathbb{E}[\mathcal{y}_j(\mathcal{q})||\pi(\rho)]=\mathbb{E}[\hat{C}(\mathcal{q})x_j(\mathcal{q})||\pi(\rho)],\label{eq4.9}
\end{align}
where in \eqref{eq4.8} and \eqref{eq4.9} $\mathbb{E}[\cdot||\pi]$ denotes expectation with respect to the measure $\pi$.

\section{Approximation and Convergence}\label{appr.conv}

In this section, we can now formally state our estimation problem and the sequence of finite dimensional approximating problems. We will also state and prove a convergence theorem.

\subsection{The Estimation Problem}\label{est.prob}

Assume that data of the form $(\{\tilde{u}_{i,j}\}_{j=0}^{n_i-1},\{\tilde{y}_{i,j}\}_{j=0}^{n_i})_{i=1}^m$, has been given. Determine 
$\rho^{*} = (\vec{a}^{*},\vec{b}^{*},\vec{\theta}^{*}) \in \Xi$, $\Xi$ a compact subset of $\mathbb{R}^{2p} \times \Theta \subset \mathbb{R}^{2p + r}$, $\vec{a}^{*}=[a_{i}^{*}]_{i=1}^{p}$, $\vec{b}^{*}=[b_{i}^{*}]_{i=1}^{p}$, which minimizes
\begin{align}
J(\rho)=\sum_{i=1}^mJ_i(\rho)=\sum_{i=1}^m\sum_{j=0}^{n_i}|\mathcal{y}_{i,j}(\{\tilde{u}_{i,k}\}_{k=0}^{n_i-1},\rho)-\tilde{y}_{i,j}|^2\label{eq4.10}
\end{align}
where for $i=1,2,...,m$, $\mathcal{y}_{i,j}(\{\tilde{u}_{i,k}\}_{k=0}^{n_i-1},\rho)$ is given by \eqref{eq4.7} with $u_j=\tilde{u}_{i,j}$, $j=0,...,n_i$, $i=1,2,...,m$, and \eqref{eq4.9}.                    

Recalling the assumption that for $i \in \{1,2,...,p\}$, $-\infty<\bar{\alpha}\leq a_i < b_i \leq \bar{\beta}<\infty$, let $\bar{Q}=\prod_{i=1}^p[\bar{\alpha},\bar{\beta}]$. Let $\bar{\rho} = ([\bar{\alpha}]_{i=1}^{p},[\bar{\beta}]_{i=1}^{p},\vec{\theta}) \in \Xi$,
$\bar{\mathcal{H}}=L_{\pi(\bar{\rho})}^2(\bar{Q};H)$ and $\bar{\mathcal{V}}=L_{\pi(\bar{\rho})}^2(\bar{Q};V)$. Then, for $N=1,2,...$, let $\vec{a}^N=[a_i^N]_{i=1}^p$, $\vec{b}^N=[b_i^N]_{i=1}^p$ be such that $-\infty<\bar{\alpha} \leq a_i^N<b_i^N \leq \bar{\beta}<\infty$, and let $\rho^N = ([\vec{a}^N,\vec{b}^N,\vec{\theta}) \in \Xi$.  Set $Q^N=\prod_{i=1}^p[a_i^N,b_i^N]$, $\mathcal{H}^N=L_{\pi(\bar{\rho}^N)}^2(Q^N,H)$, $\mathcal{V}^N=L_{\pi(\bar{\rho}^N)}^2(Q^N,V)$ and let $\mathcal{U}^N$ be a finite dimensional subspace of $\mathcal{V}^N$.  Let $\mathcal{I}^N:\bar{\mathcal{H}}\rightarrow\mathcal{H}^N$ be a linear map defined by $\mathcal{I}^N(\psi)=\psi|_{Q^N}$ for any $\psi\in\bar{\mathcal{H}}$, let $\mathcal{P}^N:\mathcal{H}^N\rightarrow\mathcal{U}^N$ denote the orthogonal projection of $\mathcal{H}^N$ onto $\mathcal{U}^N$, and define $\mathcal{J}^N:\bar{\mathcal{H}}\rightarrow\mathcal{U}^N$ by $\mathcal{J}^N =\mathcal{P}^N\circ\mathcal{I}^N$.  

In addition, recall that we have assumed that for $\rho \in \Xi$, the probability distributions described by $\pi(\rho)$ are all absolutely continuous; that is $\pi(\rho)\sim f(\rho)$, where $f(\rho)=f(\cdot;\rho)$ is a joint density for the random vector $\mathcal{q}$. 

Noting that in this formulation, $\mathcal{U}^N$ is neither a subspace of $\bar{\mathcal{H}}$ nor $\bar{\mathcal{V}}$, we define the operators $\mathcal{A}^N(\rho)$ on $\mathcal{U}^N$ to be what are essentially the restrictions of $\mathcal{A}(\rho)$ to the spaces $\mathcal{U}^N$.  More precisely, we set
\begin{align}
        \left<\mathcal{A}^N(\rho)v^N,w^N\right>=-\mathcal{a}(\rho;v^N,w^N)=-\int_{Q}a(q;v^N(q),w^N(q))d\pi(q,\rho)\nonumber\\
        =-\int_{Q}a(q;v^N(q),w^N(q))f(q;\rho)dq=-\mathbb{E}[a(\mathcal{q};v^N(\mathcal{q}),w^N(\mathcal{q}))||\pi(\rho)],\label{eq5.1}
\end{align}
where $v^N,w^N\in\mathcal{U}^N$.

Define the operators $\mathcal{B}^N(\rho):\mathbb{R}^\mu\rightarrow\mathcal{U}^N$ and $\mathcal{C}^N(\rho):\mathcal{U}^N\rightarrow\mathbb{R}^\nu$ by
\begin{align}
    <\mathcal{B}^N(\rho)u,v^N>_{\mathcal{V}^*,\mathcal{V}}&=\int_Q\left<b(q),v^N(q)\right>_{V^*,V}d\pi(q;\rho)u\nonumber\\
    &=\mathbb{E}[\left<b(\mathcal{q}),v^N(\mathcal{q})\right>_{V^*,V}||\pi(\rho)]u,\label{eq5.2}
\end{align}
\begin{align}
    \mathcal{C}^N(\rho)v^N=\int_Q\left<c(q),v^N\right>_{V^*,V}d\pi(q;\rho)=\mathbb{E}[\left<c(\mathcal{q}),v^N(\mathcal{q})\right>_{V^*,V}||\pi(\rho)],\label{eq5.3}
\end{align}
where $v^N\in\mathcal{U}^N$, and $u\in\mathbb{R}^\mu$.

With these definitions, we can now state the finite dimensional approximating problems.

Assume that data of the form $(\{\tilde{u}_{i,j}\}_{j=0}^{n_i-1},\{\tilde{y}_{i,j}\}_{j=0}^{n_i})_{i=1}^m$, has been given. Determine 
$\rho^{N*} = (\vec{a}^{N*},\vec{b}^{N*},\vec{\theta}^{N*}) \in \Xi$, $\Xi$ a compact subset of $\mathbb{R}^{2p} \times \Theta \subset \mathbb{R}^{2p + r}$, $\vec{a}^{N*}=[a_{i}^{N*}]_{i=1}^{p}$, $\vec{b}^{N*}=[b_{i}^{N*}]_{i=1}^{p}$, which minimizes
\begin{align}
J^N(\rho)=\sum_{i=1}^m\sum_{j=0}^{n_i}|\mathcal{y}_{i,j}^N(\{\tilde{u}_{i,k}\}_{k=0}^{n_i-1},\rho)-\tilde{y}_{i,j}|^2, \label{eq5.4}
\end{align}
where in \eqref{eq5.4}, for $i=1,2,...,m$, $\mathcal{y}_{i,j}^N(\{\tilde{u}_{i,k}\}_{k=0}^{n_i-1},\rho)=\hat{\mathcal{C}}(\rho)^N\mathcal{x}_{i,j}^N$ is given by \eqref{eq4.7} and \eqref{eq4.9} with $u_j=\tilde{u}_{i,j}$, $j=0,...,n_i$, $i=1,2,...,m$, $\mathcal{x}_{j}$ replaced by $\mathcal{x}_{i,j}^N\in\mathcal{U}^N$, $\hat{\mathcal{A}}(\rho)$ replaced by $$\hat{\mathcal{A}}^N(\rho)=\mathcal{T}^N(\tau;\rho) = e^{\mathcal{A}^N(\rho)\tau}\in\mathcal{L}(\mathcal{U}^N,\mathcal{U}^N),$$ $\hat{\mathcal{B}}(\rho)$ replaced by $\hat{\mathcal{B}^N}(\rho)=\int_0^{\tau}e^{\mathcal{A}^N(\rho)s}\mathcal{B}^N(\rho)ds\in\mathcal{L}(\mathbb{R}^\mu,\mathcal{U}^N)$, $\hat{\mathcal{C}}(\rho)$ replaced by $\hat{\mathcal{C}}^N(\rho)\in\mathcal{L}(\mathcal{U}^N,\mathbb{R}^\nu)$, and $\mathcal{x}_{i,0}$ replaced by $\mathcal{x}_{i,0}^N = \mathcal{J}^N\mathcal{x}_{i,0} \in \mathcal{U}^N$.  It follows that for $i=1,2,...,m$,
\begin{align}
    \mathcal{x}_{i,j+1}^N=\hat{\mathcal{A}}^N(\rho)\mathcal{x}_{i,j}^N+\hat{\mathcal{B}}^N(\rho)\tilde{u}_{i,j}, \ 
    \mathcal{y}_{i,j}^N = \mathcal{C}^N(\rho) \mathcal{x}_{i,j}^N \ \ j=0,1,2,...,\label{eq5.5}
\end{align}
with the operators $\mathcal{A}^N(\rho)$, $\mathcal{B}^N(\rho)$, and $\mathcal{C}^N(\rho)$ appearing in \eqref{eq5.5} are as they have been defined above using \eqref{eq5.1}-\eqref{eq5.3}.

In the following sections we prove that there exists a subsequence of solutions to the sequence of approximating problems that converges to the solution of our original estimation/optimization problem.

\subsection{A Version of the Trotter-Kato Semigroup Approximation Theorem}\label{trotterkato}

Our convergence proof is based on a version of the Trotter-Kato semigroup approximation theorem (\cite{BK,K,pazy83}) that does not require the approximating spaces to be subspaces of the underlying infinite dimensional state space. Banks, Burns and Cliff \cite{BBC} proved just such a result but unfortunately they do not state their hypotheses in terms of resolvent convergence which is what we require here. Consequently we establish the result in its requisite form here. 

Let $\hat{H}$ be a Hilbert space with norm $|\cdot|$ and let $\{\hat{H}^N\}$ be a sequence of Hilbert spaces, each equipped with norm $|\cdot|_N$. Assume that for each $N\in\mathbb{N}$, $\hat{U}^N$ is a closed (finite dimensional) subspace of $\hat{H}^N$.  Assume that the operators $\hat{A}$ on $\hat{H}$, and for each $N\in\mathbb{N}$, $\hat{A}^N$ on $\hat{U}^N$, are in $G(M,\lambda_0)$ with $M$ and $\lambda_0$ independent of $N$; that is they are the infinitesimal generators of $C_0$-semigroups $\hat{S}(t)$ on $\hat{H}$ and $\hat{S}^N(t)$, on $\hat{U}^N$, respectively, that are uniformly (uniformly in $N$) exponentially bounded. (We note that if $\hat{A}$ is obtained from a bounded and coercive sesquilinear form and the $\hat{U}^N$'s are subspaces with $\hat{A}^N$ defined as the restrictions of $\hat{A}$ to $\hat{U}^N$, then this latter assumption is easily verified \cite{BI,BK}.)

\begin{theorem}\label{th5.1}
Let $\hat{H}$, $\hat{H}^N$, and $\hat{U}^N$ be Hilbert spaces as defined above. Let $\mathcal{I}^N:\hat{H}\rightarrow\hat{H}^N$ be an operator such that $Im(\mathcal{I}^N)=\hat{H}^N$ and $|\mathcal{I}^Nz|_N\leq|z|$. Let $\mathcal{p}^N:\hat{H}^N\rightarrow\hat{U}^N$ be the canonical projection of $\hat{H}^N$ onto $\hat{U}^N$ and define $P^N:=\mathcal{p}^N\circ\mathcal{I}^N$. Let $\hat{A}\in G(M,\lambda_0)$ on $\hat{H}$, and $\hat{A}^N\in G(M,\lambda_0)$ on $\hat{U}^N$. Suppose that for some $\lambda\geq\lambda_0$,
\begin{align}
    |P^NR_{\lambda}(\hat{A})z-R_{\lambda}(\hat{A}^N)P^Nz|_N\rightarrow0, \ \ as \ \ N\rightarrow\infty, \label{eq5.6}
\end{align}
for every $z\in\hat{H}$, where $R_{\lambda}(\hat{A})=(\lambda I-\hat{A})^{-1}$ and $R_{\lambda}(\hat{A}^N)=(\lambda I-\hat{A}^N)^{-1}$ denote respectively the resolvent operators of $\hat{A}$ and $\hat{A}^N$ at $\lambda$. Then
\begin{align}
    |P^N\hat{S}(t)z-\hat{S}^N(t)P^Nz|_N\rightarrow0, \ \ as \ \ N\rightarrow\infty, \label{eq5.7}
\end{align}
in $\hat{H}^N$, for every $z\in\hat{H}$ uniformly in $t$ on compact $t$-intervals.
\end{theorem}

\begin{proof}
For ease of exposition and without loss of generality, let $\lambda_0=0$. Then, since $\hat{S}(t)R_{\lambda}(\hat{A})$ and $\hat{S}^N(t)R_{\lambda}(\hat{A}^N)$ are both strongly differentiable in $t$, we have
\begin{align}
    \frac{d}{dt}\hat{S}(t)R_{\lambda}(\hat{A})=\hat{A}\hat{S}(t)R_{\lambda}(\hat{A})=\hat{S}(t)\hat{A}R_{\lambda}(\hat{A})=\hat{S}(t)[\lambda R_{\lambda}(\hat{A})-I]. \label{eq5.8}
\end{align}
Then, using an identity for $\hat{S}^N(t)R_{\lambda}(\hat{A}^N)$ analogous to \eqref{eq5.8} , we obtain
\begin{align}
    \begin{aligned}
    \frac{d}{ds}[\hat{S}^N(t-s)R_{\lambda}(\hat{A}^N)P^N&\hat{S}(s)R_{\lambda}(\hat{A})]\\ &=\hat{S}^N(t-s)[P^NR_{\lambda}(\hat{A})-R_{\lambda}(\hat{A}^N)P^N]\hat{S}(s).\label{eq5.9}
    \end{aligned}
\end{align}
Then, since 
\begin{align}
    \begin{aligned}
    \hat{S}^N(t-s)R_{\lambda}(\hat{A}^N)P^N&\hat{S}(s)R_{\lambda}(\hat{A})|_{s=0}^{s=t}\\
    &=R_{\lambda}(\hat{A}^N)[P^N\hat{S}(t)-\hat{S}^N(t)P^N]R_{\lambda}(\hat{A}),\label{eq5.10}
    \end{aligned}
\end{align}
\eqref{eq5.9} and \eqref{eq5.10} imply that
\begin{align}
    \begin{aligned}
    R_{\lambda}(\hat{A}^N)[P^N\hat{S}(t)-&\hat{S}^N(t)P^N]R_{\lambda}(\hat{A})\\
    &=\int_0^t\hat{S}^N(t-s)[P^NR_{\lambda}(\hat{A})-R_{\lambda}(\hat{A}^N)P^N]\hat{S}(s)ds.\label{eq5.11}
    \end{aligned}
\end{align}
Equation \eqref{eq5.11} and $|\hat{S}^N(t-s)|\leq M$ (recall $\lambda_0=0$), for any $u\in\hat{H}$, yield
\begin{align}
    \begin{aligned}
    |R_{\lambda}(\hat{A}^N)[P^N\hat{S}(t)-&\hat{S}^N(t)P^N]R_{\lambda}(\hat{A})u|_N\\
    &\leq M\int_0^t|[P^NR_{\lambda}(\hat{A})-R_{\lambda}(\hat{A}^N)P^N]\hat{S}(s)u|_Nds.\label{eq5.12}
    \end{aligned}
\end{align}
By \eqref{eq5.6}, we know that the integrand in \eqref{eq5.12} converges to 0 for a fixed $s$, and also it is bounded by $2M^2|u|/\lambda$, and therefore, by the Lebesque Dominated Convergence Theorem, the right-hand side of \eqref{eq5.12} converges to 0 as $N\rightarrow\infty$, where the convergence is uniform in $t$ on compact $t$-intervals.

Letting $v=R_{\lambda}(\hat{A})u$, and using the fact that $D(\hat{A})$ is dense in $\hat{H}$, we have that
\begin{align}
    |R_{\lambda}(\hat{A}^N)[P^N\hat{S}(t)-\hat{S}^N(t)P^N]v|_N\rightarrow0, \ \ as \ \ N\rightarrow\infty, \label{eq5.13}
\end{align}
for all $v\in\hat{H}$. Then, since $|\hat{S}(t)|\leq M$, \eqref{eq5.6} implies that
\begin{align}
    \begin{aligned}
    |R_{\lambda}(\hat{A}^N)\hat{S}^N(t)P^Nv-&\hat{S}^N(t)P^NR_{\lambda}(\hat{A})v|_N\\
    &=|\hat{S}^N(t)[R_{\lambda}(\hat{A}^N)P^Nv-P^NR_{\lambda}(\hat{A})v]|_N\rightarrow0,\label{eq5.14}
    \end{aligned}
\end{align}
and similarly, $|\hat{S}^N(t)|\leq M$ and (\eqref{eq5.6}) imply that
\begin{align}
\begin{aligned}
    |R_{\lambda}(\hat{A}^N)P^N\hat{S}(t)v-&P^N\hat{S}(t)R_{\lambda}(\hat{A})v|_N\\
    &=|[R_{\lambda}(\hat{A}^N)P^N-P^NR_{\lambda}(\hat{A})]\hat{S}(t)v|_N\rightarrow0. \label{eq5.15}
    \end{aligned}
\end{align}
Combining \eqref{eq5.14}, \eqref{eq5.15}, and the triangle inequality we get
\begin{align}
    |R_{\lambda}(\hat{A}^N)[\hat{S}^N(t)P^N-P^N\hat{S}(t)]v+[P^N\hat{S}(t)-\hat{S}^N(t)P^N]R_{\lambda}(\hat{A})v|_N\rightarrow0, \label{eq5.16}
\end{align}
as $N\rightarrow\infty$. Then, because of \eqref{eq5.13}, and again by the triangle inequality, we obtain that
\begin{align}
|[P^N\hat{S}(t)-\hat{S}^N(t)P^N]R_{\lambda}(\hat{A})v|_N\rightarrow0, \ \ as \ \ N\rightarrow\infty. \label{eq5.17}
\end{align}
Letting $w=R_{\lambda}(\hat{A})v$, we have $w\in Dom(\hat{A}^2)$; and since $Dom(\hat{A}^2)$ is dense in $\hat{H}$, it follows from \eqref{eq5.6}, \eqref{eq5.16} and \eqref{eq5.17} that
\begin{align}
    |\hat{S}^N(t)P^Nz-P^N\hat{S}(t)z|_N\rightarrow0, \ \ as \ \ N\rightarrow\infty, \nonumber 
\end{align}
for all $z\in\hat{H}$ uniformly in $t$ on compact $t$-intervals.
\end{proof}

\subsection{Application to the Density Estimation Problem}\label{app.prob.est}

Let $\{\rho^N\}, \rho \in \Xi$ be such that $f^N(q)\rightarrow f(q)$, for almost every $q\in Q$, where $f^N(q) = f(q;\rho^N)$ and $f(q) = f(q;\rho)$. Let $\bar{\mathcal{H}}$, $\bar{\mathcal{V}}$, $\mathcal{H}^N$, $\mathcal{V}^N$, $\mathcal{U}^N$, $\mathcal{I}^N:\bar{\mathcal{H}}\rightarrow\mathcal{H}^N$, $\mathcal{P}^N:\mathcal{H}^N\rightarrow\mathcal{U}^N$, and  $\mathcal{J}^N:\bar{\mathcal{H}}\rightarrow\mathcal{U}^N$ be as they were defined earlier. Set $\mathcal{A} = \mathcal{A}(\rho)$ and consider it to be an operator on $\bar{\mathcal{H}}$ and $\bar{\mathcal{V}}$ by extending $f(\cdot,\rho)$, which is defined on $Q$, to $\bar{Q}$ by setting it equal to zero on $\bar{Q} \setminus Q$ and let $\mathcal{A^N} = \mathcal{A}^N(\rho^N)$. Then it follows from Assumptions (i) - (iii) that $\mathcal{A}$ is in $G(M,\lambda_0)$ on $\bar{\mathcal{H}}$ and $\mathcal{A^N}$ is in $G(M,\lambda_0)$ on $\mathcal{H}^N$ with $M$ and $\lambda_0$ independent of $N$.

In the statement of Theorem \eqref{th5.1}, set $\hat{H} = \bar{\mathcal{H}}$, $\hat{H}^N=\mathcal{H}^N$, $\hat{U}^N = \mathcal{U}^N$, $P^N=\mathcal{J}^N$, $\hat{A} = \mathcal{A}$, and $\hat{A}^N=\mathcal{A}^N$. To apply Theorem \eqref{th5.1} and conclude that in this case, \eqref{eq5.7} holds, we need only verify \eqref{eq5.6}. In order to do this, we require the following two additional assumptions

\begin{enumerate}[iv.]
\item There exist positive real numbers $\gamma$ and $\delta$ such that for any $\rho \in \Xi$, we have $0<\gamma\leq f(q;\rho)\leq\delta<\infty$ for $\pi(\rho)$-a.e. $q\in Q$.
\end{enumerate}

\begin{enumerate}[v.]
\item For all $w\in\bar{\mathcal{V}}$, there exists $u^N\in\mathcal{U}^N$ such that  $ ||u^N-\mathcal{J}^Nw||_{\mathcal{V}^N}\rightarrow0$ as $N\rightarrow\infty$.
\end{enumerate}

We are now able to prove the following theorem.

\begin{theorem}\label{th5.2}
Let assumptions (i) - (v) be satisfied and let $\{\rho^N\}, \rho \in \Xi$ be such that $f^N(q)\rightarrow f(q)$, for almost every $q\in \bar{Q}$, where  $f^N(q) = f(q;\rho^N)$ and $f(q) = f(q;\rho)$. Then, with the definitions above, the conditions of Theorem \eqref{th5.1} (and in particular the resolvent convergence specified in \eqref{eq5.6}) are satisfied. Consequently, it follows that
\begin{align}
    ||\mathcal{T}^N(t;\rho^N)P^Nz-\mathcal{J}^N\mathcal{T}(t;\rho)z||_{\mathcal{H}^N}\rightarrow0, \ \ as \ \ N\rightarrow\infty,\label{eq5.18}
\end{align}
for every $z\in\mathcal{H}$, uniformly in t on compact t-intervals where $\mathcal{T}^N=\{\mathcal{T}^N(t;\rho^N):t\geq0\}$ is the semigroup on $\mathcal{H}^N$ given by $\mathcal{T}^N(t;\rho^N)=e^{\mathcal{A}^Nt}=e^{\mathcal{A}^N(\rho^N)t}$ and $\mathcal{T}=\{\mathcal{T}(t;\rho):t\geq0\}$ is the semigroup on $\mathcal{H}$ and $\bar{\mathcal{H}}$ given by $\mathcal{T}(t;\rho)=e^{\mathcal{A}t}=e^{\mathcal{A}(\rho)t}$.
\end{theorem}

\begin{proof}
First, note that if we can show resolvent convergence for every $z\in\bar{\mathcal{V}}$, then since $\bar{\mathcal{V}}$ is dense in $\bar{\mathcal{H}}$, and $\mathcal{J}^NR_{\lambda_0}(\mathcal{A})$ and $R_{\lambda_0}(\mathcal{A}^N)\mathcal{J}^N$ are uniformly bounded, the desired resolvent convergence for every $z\in\bar{\mathcal{H}}$ will have been demonstrated. In what follows, for any $\rho = (\vec{a},\vec{b},\vec{\theta})\in \Xi$, $f(\cdot;\rho)$ is defined on $Q=\prod_{i=1}^p[a_i,b_i]$, but it can be extended to be defined on $\bar{Q}$ by setting it equal to zero on $\bar{Q} \setminus Q$. We will use this fact frequently below without further remark.

Let $z\in\bar{\mathcal{V}}$ and define $w=R_{\lambda_0}(\mathcal{A})z$, and $w^N=R_{\lambda_0}(\mathcal{A}^N)\mathcal{J}^Nz$. Suppose also that $u^N\in\mathcal{U}^N$ be as in Assumption (v) for $w=R_{\lambda_0}(\mathcal{A})z$.

Then, by triangle inequality, we have
\begin{align}
    \begin{aligned}
        ||\mathcal{J}^Nw-w^N||_{\mathcal{V}^N}&\leq||\mathcal{J}^Nw-u^N+u^N-w^N||_{\mathcal{V}^N}\\
        &\leq||\mathcal{J}^Nw-u^N||_{\mathcal{V}^N}+||u^N-w^N||_{\mathcal{V}^N}.\label{eq5.19}
    \end{aligned}
\end{align}
Thus, \eqref{eq5.19}, Assumption (v) and the continuous embedding of $\mathcal{V}^N$ in $\mathcal{H}^N$ imply that it is enough to show that $||u^N-w^N||_{\mathcal{V}^N}\rightarrow0$ as $N\rightarrow\infty$. Let $z^N=w^N-u^N$. Then, since $w^N\in \mathcal{U}^N \subset \mathcal{V}^N)$,
\begin{align}
    \begin{aligned}
        \mathcal{a}(\rho^N;w^N,z^N)&=\left<-\mathcal{A}^Nw^N,z^N\right>_{\mathcal{H}^N}\\
        &=\left<(\lambda_0I-\mathcal{A}^N)R_{\lambda_0}(\mathcal{A}^N)\mathcal{J}^Nz,z^N\right>_{\mathcal{H}^N}-\lambda_0\left<w^N,z^N\right>_{\mathcal{H}^N}\\
        &=\left<\mathcal{J}^Nz,z^N\right>_{\mathcal{H}^N}-\lambda_0\left<w^N,z^N\right>_{\mathcal{H}^N}.\label{eq5.20}
    \end{aligned}
\end{align}
Also, since $w\in Dom(\mathcal{A})$,
\begin{align}
    \begin{aligned}
        \mathcal{a}(\rho;w,\mathcal{I}^{N^+}z^N)&=\left<-\mathcal{A}w,\mathcal{I}^{N^+}z^N\right>_{\bar{\mathcal{H}}}\\
        &=\left<(\lambda_0I-\mathcal{A})R_{\lambda_0}(\mathcal{A})z,\mathcal{I}^{N^+}z^N\right>_{\bar{\mathcal{H}}}-\lambda_0\left<w,\mathcal{I}^{N^+}z^N\right>_{\bar{\mathcal{H}}}\\
        &=\left<z,\mathcal{I}^{N^+}z^N\right>_{\bar{\mathcal{H}}}-\lambda_0\left<w,\mathcal{I}^{N^+}z^N\right>_{\bar{\mathcal{H}}},\label{eq5.21}
    \end{aligned}
\end{align}
where $\mathcal{I}^{N^+}$ denotes the Moore-Penrose generalized inverse \cite{GILT} of $\mathcal{I}^{N}$. We note that for $\psi \in \mathcal{H}^N$,  $\mathcal{I}^{N^+}\psi$ is the function in $\bar{\mathcal{H}}$ that agrees with $\psi$ on $Q^N$ and is zero on $\bar{Q} \setminus Q^N$. Then, from \eqref{eq5.20} and \eqref{eq5.21}, we obtain
\begin{align}
    \begin{aligned}
        \mathcal{a}(\rho^N;w^N,z^N)-\mathcal{a}(\rho;w,\mathcal{I}^{N^+}z^N)=&\left<\mathcal{I}^Nz,z^N\right>_{\mathcal{H}^N}-\lambda_0\left<w^N,z^N\right>_{\mathcal{H}^N}\\
        &-\left<z,\mathcal{I}^{N^+}z^N\right>_{\bar{\mathcal{H}}}+\lambda_0\left<w,\mathcal{I}^{N^+}z^N\right>_{\bar{\mathcal{H}}}. \label{eq5.22}
    \end{aligned}
\end{align}
Recalling Assumptions (i) and (ii) for the form $a(\cdot;\cdot,\cdot)$ on $V \times V$, let $\tilde{\alpha}_0$, $\tilde{\mu}_0$, $\tilde{\lambda}_0$ denote the boundedness and coercivity coefficients for the forms $\mathcal{a}(\cdot;\cdot,\cdot)$.  Then, using boundedness, coercivity, Assumptions (iv) and (v), Young's and the Cauchy Schwarz Inequalities, and the continuous embeddings of the space $V$ in the space $H$ (i.e. that there exist a constant $k$ such that $|\cdot|_H \leq k||\cdot||_V$) and \eqref{eq5.22}, for any $\varepsilon >0$, we obtain
\begingroup
\allowdisplaybreaks
\begin{align}
    \tilde{\mu}_0&||z^N||_{\mathcal{V}^N}\leq\mathcal{a}(\rho^N;z^N,z^N)+\tilde{\lambda}_0|z^N|_{\mathcal{H}^N}\nonumber\\
    &=\mathcal{a}(\rho^N;w^N,z^N)-\mathcal{a}(\rho^N;u^N,z^N)+\tilde{\lambda}_0|z^N|_{\mathcal{H}^N}^2\nonumber\\
    &=\mathcal{a}(\rho^N;w^N,z^N)-\mathcal{a}(\rho;w,\mathcal{I}^{N^+}z^N)\nonumber\\
        &\quad+\mathcal{a}(\rho;w,\mathcal{I}^{N^+}z^N)-\mathcal{a}(\rho^N;u^N,z^N)+\tilde{\lambda}_0|z^N|_{\mathcal{H}^N}^2\nonumber\\
    &=\left<\mathcal{I}^Nz,z^N\right>_{\mathcal{H}^N}-\tilde{\lambda}_0\left<w^N,z^N\right>_{\mathcal{H}^N}-\left<z,\mathcal{I}^{N^+}z^N\right>_{\bar{\mathcal{H}}}+\tilde{\lambda}_0\left<w,\mathcal{I}^{N^+}z^N\right>_{\bar{\mathcal{H}}}\nonumber\\
        &\quad+\int_{\bar{Q}}(a(q;w,z^N)f(q)-a(q;u^N,z^N)f^N(q))dq+\tilde{\lambda}_0|z^N|_{\mathcal{H}^N}^2\nonumber\\
    &=\int_{\bar{Q}}(\left<z,z^N\right>_H(f^N(q)-f(q))dq\nonumber\\
        &\quad+\tilde{\lambda}_0\int_{\bar{Q}}(\left<w,z^N\right>_Hf(q)-\left<u^N,z^N\right>_Hf^N(q))dq\nonumber\\
        &\quad+\int_{\bar{Q}}(a(q;w,z^N)f(q)-a(q;u^N,z^N)f^N(q))dq\nonumber\\
    &=\int_{\bar{Q}}(\left<z,z^N\right>_H(f^N(q)-f(q))dq+\tilde{\lambda}_0\int_{\bar{Q}}\left<w,z^N\right>_H(f(q)-f^N(q))dq\nonumber\\
        &\quad+\tilde{\lambda}_0\int_{\bar{Q}}\left<w-u^N,z^N\right>_Hf^N(q)dq+\int_{\bar{Q}}(a(q;w,z^N)(f(q)-f^N(q))dq\nonumber\\
        &\quad+\int_{\bar{Q}}a(q;w-u^N,z^N)f^N(q)dq\nonumber\\
    &\leq\int_{\bar{Q}}|z|_H \ |z^N|_H \ |f^N(q)-f(q)|dq+\tilde{\lambda}_0\int_{\bar{Q}}|w|_H \ |z^N|_H \ |f(q)-f^N(q)|dq\nonumber\\
        &\quad+\tilde{\lambda}_0\int_{\bar{Q}}|w-u^N|_H \ |z^N|_Hf^N(q)dq+\alpha_0\int_{\bar{Q}}||w||_V \ ||z^N||_V \ |f(q)-f^N(q)|dq\nonumber\\
        &\quad+\alpha_0\int_{\bar{Q}}||w-u^N||_V \ ||z^N||_Vf^N(q)dq\nonumber\\
    &\leq\frac{\varepsilon k^2}{2\alpha}\int_{Q^N}||z^N||_V^2f^N(q)dq+\frac{1}{2\varepsilon}\int_{\bar{Q}}|z|_H^2|f^N(q)-f(q)|^2dq\nonumber\\
        &\quad+\frac{\tilde{\lambda}_0\varepsilon k^2}{2\alpha}\int_{Q^N}||z^N||_V^2f^N(q)dq+\frac{\tilde{\lambda}_0 k^2}{2\varepsilon}\int_{\bar{Q}}||w||_V^2|f^N(q)-f(q)|^2dq\nonumber\\    
        &\quad+\frac{\tilde{\lambda}_0\varepsilon k^2}{2}\int_{Q^N}||z^N||_V^2f^N(q)dq+\frac{\tilde{\lambda}_0 k^2}{2\varepsilon}\int_{Q^N}||w-u^N||_V^2f^N(q)dq\nonumber\\
        &\quad+\frac{\alpha_0\varepsilon}{2\alpha}\int_{Q^N}||z^N||_V^2f^N(q)dq+\frac{\alpha_0}{2\varepsilon}\int_{\bar{Q}}||w||_V^2|f^N(q)-f(q)|^2dq\nonumber\\
        &\quad+\frac{\alpha_0\varepsilon}{2}\int_{Q^N}||z^N||_V^2f^N(q)dq+\frac{\alpha_0}{2\varepsilon}\int_{Q^N}||w-u^N||_V^2f^N(q)dq.\label{eq5.23}
\end{align}
\endgroup
Then, letting $\tilde{c} = \tilde{\mu}_0-\frac{\varepsilon}{2\alpha}\left(k^2(\tilde{\lambda}_0 + 1)+(\alpha+1)(\tilde{\lambda}_0 k^2+\alpha_0)\right)$, it follows from \eqref{eq5.23} that
\begin{align}
    \begin{aligned}
    \tilde{c}||z^N||_{\mathcal{V}^N}^2\leq\frac{1}{2\varepsilon}\int_{\bar{Q}}|z|_H^2|f^N(q)-&f(q)|^2dq+\frac{\tilde{\lambda}_0 k^2+\alpha_0}{2\varepsilon}\int_{Q^N}||w-u^N||_V^2f^N(q)dq\\
    &+\frac{\tilde{\lambda}_0 k^2+\alpha_0}{2\varepsilon}\int_{\bar{Q}}||w||_V^2|f^N(q)-f(q)|^2dq\\
    =\frac{1}{2\varepsilon}\int_{\bar{Q}}|z|_H^2|f^N(q)-&f(q)|^2dq+\frac{\tilde{\lambda}_0 k^2+\alpha_0}{2\varepsilon}||\mathcal{J}^Nw - u^N||_{\mathcal{V}^N}^2 \\
    &+\frac{\tilde{\lambda}_0 k^2+\alpha_0}{2\varepsilon}\int_{\bar{Q}}||w||_V^2|f^N(q)-f(q)|^2dq.\label{eq5.24}
    \end{aligned}
\end{align}
Choosing $\varepsilon$ positive, but sufficiently small in \eqref{eq5.24}, it follows from Assumption (v) and the hypotheses of the theorem that
\begin{align}
    ||w^N-u^N||_{\mathcal{V}^N}=||z^N||_{\mathcal{V}^N}\rightarrow0 \ \ as \ \ N\rightarrow\infty.\label{eq5.25}
\end{align}
Thus \eqref{eq5.25} together with \eqref{eq5.19}, and Assumption (v) yield resolvent convergence and the theorem is proved.
\end{proof}

We note that in the proof of Theorem \eqref{th5.2} we were in fact able to establish resolvent convergence in the $\mathcal{V}^N$ norm. Consequently we may conclude that the semigroup convergence in \eqref{eq5.18} is in the $\mathcal{V}^N$ norm as well.  Moreover, it is not difficult to establish the following corollary to Theorem \eqref{th5.2}.

\begin{corollary}
\label{co5.1}
Under the same hypotheses of Theorem \eqref{th5.2}, we have
\begin{align}
    \begin{aligned}
    ||\mathcal{x}_{i,j}^N(\rho^N) -\mathcal{J}^N\mathcal{x}_{i,j}(\rho)||_{\mathcal{V}^N}\rightarrow0, \ \ as \ \ N\rightarrow\infty,\\
    |\mathcal{y}_{i,j}^N(\rho^N) -\mathcal{y}_{i,j}(\rho)|_{\mathbb{R}^{\nu}}\rightarrow0, \ \ as \ \ N\rightarrow\infty,\label{eq5.27}
    \end{aligned}
\end{align}
for every $i=1,2,...,m$, uniformly in $j$, for $j = 0,1,2,...,n_i$, where $\mathcal{x}_{i,j}^N(\rho^N)$ and $\mathcal{y}_{i,j}^N(\rho^N)$ are given in \eqref{eq5.5} and $\mathcal{x}_{i,j}(\rho)$ and $\mathcal{y}_{i,j}(\rho)$ are given in \eqref{eq4.7}.
\end{corollary}

The assumption that the feasible parameter set $\Xi$ is closed and bounded in $\mathbb{R}^{2p+r}$, together with \eqref{eq5.27} in the statement of Corollary \eqref{co5.1} and Theorem \eqref{th2.1} then yield the following result. 

\begin{theorem}
\label{th5.3}
If, in addition to Assumptions (i)-(v), we assume that the maps $\rho \mapsto f(q;\rho)$ from $\Xi$ to $\mathbb{R}$ are continuous for $\pi(\rho)$ a.e. $q \in \bar{Q}$, then each of the approximating estimation problems admits a solution, $\rho^{N*}$. Moreover, the sequence $\{\rho^{N*}\}$ has a convergent subsequence, $\{\rho^{N_{k}*}\}$ with $\rho^{N_{k}*}\rightarrow \rho^*$ and $\rho^*$ a solution to the original estimation problem.
\end{theorem}

It is also possible to establish a consistency result for the estimator $\rho^{*} = (\vec{a}^{*},\vec{b}^{*},\vec{\theta}^{*}) \in \Xi$.  We require the following additional assumptions:

(a)	The measurement noise $\{\varepsilon_{j,i}\}$ is i.i.d. with respect to a probability space $\{\Omega,\Sigma,P\}$ with $\mathbb{E}[\varepsilon_{j,i}||P]=0$ and $Var[\varepsilon_{j,i}||P]=\sigma^2$,

(b)	The feasible set of parameters $\Xi$ is compact (i.e. closed and bounded since it is finite dimensional) and has nonempty interior,

(c)	For $i=1,2, . . .$, $n_i = n$ and $n\tau = T$ for some positive integer $n$ and some $T > 0$, where $\tau$ is the sampling time defined in Section 3,

(d)	That $\tilde{y}_{i,j} = \mathcal{y}_{i,j}(\{\tilde{u}_{i,k}\}_{k=0}^{n_i-1},\rho_0) + \varepsilon_{j,i}$, for some $\rho_0 \in int\{\Xi\}$, where for $i=1,2,...,m$, $\mathcal{y}_{i,j}(\{\tilde{u}_{i,k}\}_{k=0}^{n_i-1},\rho)$ is given by \eqref{eq4.7} with $u_j=\tilde{u}_{i,j}$, $j=0,...,n_i$, $i=1,2,...,m$, and \eqref{eq4.9}, and

(e)	For each $i = 1,2, . . .,m$, $\rho_0 \in \Xi$ is the unique minimizer of $J_{i,0}$ in $Xi$ where
\begin{align}
             J_{i,0}(\rho) = \sigma^2 + \int_0^T(\mathcal{y}(t;\tilde{u}_{i},\rho_0) - \mathcal{y}(t;\tilde{u}_{i},\rho))^2dt, \label{eq5.28}
\end{align}
and  $\mathcal{y}(t;\tilde{u}_{i},\rho)$ is given by \eqref{eq4.4} -\eqref{eq4.6}with $u = \tilde{u}_i$. 

Then a straight forward application of Theorem 4.2 in \cite{BT} can then be used to establish the following lemma and theorem (see \cite{SR2017}).

\begin{lemma}\label{l5.1}
If in addition to Assumptions (i)-(iv) and (a) – (e) above we assume that the maps $\rho \mapsto f(q;\rho)$ from $\Xi$ to $\mathbb{R}$ are continuous for $\pi(\rho)$ a.e. $q \in \bar{Q}$, then there exists an event $A \in \Sigma$ with $P(A) = 1$ such that for all $\omega \in A$ and $J$ as given in \eqref{eq4.10} we have
\begin{align}
              \frac{1}{m}\sum_{i=1}^{n}\{\frac{1}{n}J_i(\rho) - J_{i,0}(\rho) \} \rightarrow 0,\nonumber
\end{align}
as $n,m \rightarrow \infty$ and $\tau \rightarrow 0$, with $n\tau = T$, uniformly in $\rho$ for $\rho \in \Xi$, where $J_i$ is given by \eqref{eq4.10} and $J_{i,0}$ by \eqref{eq5.28}.
\end{lemma}

\begin{theorem}\label{th5.4}
(Consistency of the estimator $\rho^*$) Let $\rho^* \in \Xi$ be as defined in \eqref{eq4.10} in Section 5.1. Then under the assumptions of Lemma \eqref{l5.1}the estimator $\rho^{*} = (\vec{a}^{*},\vec{b}^{*},\vec{\theta}^{*}) \in \Xi$ is consistent for $\rho_0$. That is $\rho^* \ rightarrow \rho_0$ in probability with repsect to the probability measure $P$, as $m,n \rightarrow \infty$, and $\tau \rightarrow 0$ with $n\tau = T$.
\end{theorem}

\section{Examples and Numerical Results}\label{num.res}

\subsection{The Adjoint Method}\label{notes.num.appr}

The approximating optimization problems are solved numerically by using an iterative gradient-based scheme. Once a basis for the space $^N$ is chosen, matrix forms of the operators $\hat{\mathcal{A}}^N$, $\hat{\mathcal{B}}^N$, and $\hat{\mathcal{C}}^N$ can be computed. The gradient of $J^N(\rho)$, with respect to the $2p + r$ parameters in $\rho$ can be computed accurately (in fact exactly with the exception of finite precision arithmetic round-off) and efficiently (which is especially important if the dimension of the approximating system \eqref{eq5.5} and/or the number of parameters is large) using the adjoint method (see \cite{LR}). For each $i=1,...,m$, set $v_{i,j}^N=2[\hat{\mathcal{C}}^N]^T(\hat{\mathcal{C}}^N\mathcal{x}_{i,j}^N-\tilde{y}_{i,j}) \in\mathbb{R}^{K^N}$, $j=0,...,n_i$ where $K^N$ is the number of basis elements for $\mathcal{U}^N$. Then for each $i=1,...,m$, the adjoint systems are defined to be
\begin{align}
    z_{i,j-1}^N=[\hat{\mathcal{A}}^N]^Tz_{i,j}^N+v_{i,j-1}^N, z_{i,n_i} = v_{i,n_i}^N, j=n_i,n_i-1,...,2,1.\label{eq6.1}
\end{align}
The gradient of $J^N$ at $\rho = (\vec{a},\vec{b},\vec{\theta})$ can then be computed from
%
\begin{align}
    \begin{aligned}
        \vec{\nabla}J^N(\rho)&=\sum_{i=1}^m\sum_{j=1}^{n_i}[z_{i,j}^N]^T\biggl(\frac{\partial\hat{\mathcal{A}}^N}{\partial\rho}\mathcal{x}_{i,j-1}^N\\
        &\quad-(\mathcal{A}^N)^{-1}\biggl(\frac{\partial\mathcal{A}^N}{\partial\rho}(\mathcal{A}^N)^{-1}(\hat{\mathcal{A}}^N-I)\mathcal{B}^N\tilde{u}_{i,j-1}\\
        &\quad-\frac{\partial\hat{\mathcal{A}}^N}{\partial\rho}\mathcal{B}^N\tilde{u}_{i,j-1}-(\hat{\mathcal{A}}^N-I)\frac{\partial\mathcal{B}^N}{\partial\rho}\tilde{u}_{i,j-1}\biggr)\biggr)\\
        &\quad+\sum_{i=1}^m\sum_{j=0}^{n_i}\bigl(\mathcal{y}_j^N-\tilde{y}_{i,j}\bigr)^T\frac{\partial\hat{\mathcal{C}}^N}{\partial\rho}\mathcal{x}_{i,j}^N.\label{eq6.2}
    \end{aligned}
\end{align}
Using \eqref{eq6.1} and \eqref{eq6.2} to compute the gradient requires the calculation of the tensor $\frac{\partial\hat{\mathcal{A}}^N}{\partial\rho}$. This can be done using the sensitivity equations. For $t\geq0$ set $\Phi^N(t)=e^{\mathcal{A}^N(t)}$ from which differentiation yields
\begin{align}
    \begin{aligned}
        \dot{\Phi}^N(t)=\mathcal{A}^N\Phi^N(t), \ \ \Phi^N(0)=I.\label{eq6.3}
    \end{aligned}
\end{align}
Then, setting $\Psi^N(t)=\frac{\partial\Phi^N(t)}{\partial\rho}$, differentiating \eqref{eq6.3} with respect to $\rho$, and interchanging the order of differentiation, we obtain
\begin{align}
    \begin{aligned}
        \dot{\Psi}^N(t)=\mathcal{A}^N\Psi^N(t)+\frac{\partial\mathcal{A}^N}{\partial\rho}\Phi^N(t), \ \ \Psi^N(0)=0.\label{eq6.4}
    \end{aligned}
\end{align}
Combining \eqref{eq6.3} and \eqref{eq6.4}, and solving the resulting system, we obtain
\begin{align}
\begin{bmatrix}
\Psi^N(t)\\
\Phi^N(t)
\end{bmatrix}
=exp\biggl(
\begin{bmatrix}
\mathcal{A}^N & \partial\mathcal{A}^N/\partial\rho\\
0 & \mathcal{A}^N
\end{bmatrix}
\tau\biggr)
\begin{bmatrix}
0\\
I
\end{bmatrix}\label{eq6.5}
\end{align}
Setting $t=\tau$ in \eqref{eq6.5}, we obtain that $\frac{\partial\hat{\mathcal{A}}^N}{\partial\rho}=\Psi^N(\tau)$.

To illustrate our approach, we consider the case of a one dimensional heat/diffusion equation on the interval $[0,1]$ with random (thermal) diffusivity and two different sets of boundary conditions. Consider the partial differential equation, boundary conditions and output operator given by
\begingroup
\allowdisplaybreaks
\begin{align}
    \frac{\partial x}{\partial t}(t,\eta)&=q_1\frac{\partial^2 x}{\partial \eta^2}(t,\eta), \ \ 0<\eta<1, \ \ t>0,\label{eq6.6}\\
    \Gamma_Dx(t,\cdot)&=x(t,0)=0, \ \ t>0,\label{eq6.7}\\
    \Gamma_Rx(t,\cdot)&=q_1\frac{\partial x}{\partial\eta}(t,0)-x(t,0)=0, \ \ t>0,\label{eq6.8}\\
    \Gamma_1x(t,\cdot)&=\frac{q_1}{q_2}\frac{\partial x}{\partial\eta}(t,1)=u(t) \ \ t>0,\label{eq6.9}\\
    x(0,\eta)&=0, \ \ 0<\eta<1,\label{eq6.10}\\
    y(t)&=x(t,\eta_0), \ \ t>0,\label{eq6.11}
\end{align}
\endgroup
where $0 < \eta_0 < 1.$ In the examples below, we consider the parameterized family of probability density functions defined as follows.

\begin{definition}\label{trun.exp.fam} Let $\varphi(q;\theta)$, $q\in\mathbb{R}^n$ be a member in an exponential family \cite{CB}, and let $\Phi$ denote its cumulative distribution function. Let $\theta$ represent a vector of parameters, and let $D\subset\mathbb{R}^n$ be a bounded region to which $\varphi$ will be restricted. Then define $\Phi_D(\theta)=\int_D\varphi(q;\theta)dq$. Then the family of pdfs, $f(\cdot,\rho)$ given by 
\begin{align*}
    f(q;\rho)=\frac{\varphi(q;\theta)\chi_D(q)}{\Phi_D(\theta)}=\frac{1}{\Phi_D(\theta)}h(q)c(\theta)exp\left(\sum_{i=1}^kw_i(\theta)t_i(q)\right)\chi_D(q)
\end{align*}
where the parameters $\rho$ include the parameters $\theta$ and parameters $\vec{a}$ and $\vec{b}$ to describe the domain $D$, is called a \textit{truncated exponential family}.
\end{definition}
It is clear that this family of densities satisfies Assumption (iv) and the hypotheses of Theorem \eqref{th5.1}. 

All of the numerical results presented here use simulation data. Our studies involving actual experimental/clinical data are discussed elsewhere (see \cite{SR2017}). The simulated data was generated by first sampling the target distribution to obtain 100 samples $q$ of $\mathcal{q}$. A spline based Galerkin approximation to the system \eqref{eq6.6} -\eqref{eq6.11} using a 128 equally spaced point grid on $[0,1]$ was then solved using each $\mathcal{q}$-sample. The resulting 100 output signals were then averaged at each time point. The approximating estimation problems were all solved on either MAC or PC laptops using the Matlab optimization toolbox routine FMINCON for constrained optimization. Gradients were computed using either FMINCON built-in finite differencing or the adjoint method, \eqref{eq6.1}-\eqref{eq6.5}. Which method was used had only a negligible effect on the results. The input signal used was $u(t) = |cos(t)|\chi_{[0,2]}(t)$, $t \in [0,20]$, and the sampling interval was $\tau = 0.1$.  In all of our examples below, the admissible parameter space $Q$ is assumed to be either in $\mathbb{R}^+$ in the case of the uni-variate examples, or in the fist quadrant of the plane $\mathbb{R}^2$ in the bivariate examples.  Consequently when the approximating optimization problems were solved, the lower bounds for the supports of the random parameters, $a$ and $c$, were constrained to be strictly positive. This is based on the requirements of the physical model \eqref{eq6.6}-\eqref{eq6.11} and the assumption that properties (i)-(iii) in Section 3 hold.

\subsection{Examples 6.1,6.2 and 6.3; One Random Parameter; Truncated Uniform, Exponential and Normal Distributions}\label{first.examp}

In this series of examples we consider the system \eqref{eq6.6},\eqref{eq6.7},\eqref{eq6.9}-\eqref{eq6.11} with $q_1$ random and $q_2=1$. In this case we have $q=q_1\in Q=[a,b]$, $W=\{\varphi\in H^2(0,1), \Gamma_D\varphi=0\}$, $H=H_L^1(0,1)=\{\varphi\in H^1(0,1),\Gamma_D\varphi=0\}$, $Dom(A(q))=\{\varphi\in V: \Gamma_1\varphi=0\}$, and $\Gamma(q)=\Gamma_1$. It follows that
\begin{align*}
    a(q;\varphi,\psi)=q\int_0^1\varphi'(\eta)\psi'(\eta)d\eta, \ \ \varphi,\psi\in V,
\end{align*}
and $\left<b(q),\psi\right>_{V^*,V}=\left<b,\psi\right>_{V^*,V}=\psi(1)=\delta(\cdot-1), \ \ \psi\in V$, and $ \left<c(q),\psi\right>_{V^*,V}=\left<c,\psi\right>_{V^*,V}=\psi(1/3), \ \ \psi\in V$, where in this case $\eta_0 = 1/3$. Standard arguments \cite{BI,BK} show that Assumptions (i)-(iii) are satisfied.

To carry out the finite dimensional discretization, we let $n,m$ be positive integers and set $N=(n,m)$. In this case we have either $D=[a,b]$ (uniform and normal) or $D=[0,R]$ (exponential). In what follows we describe the $q$ or $Q$ discretization for the uniform and normal cases; the exponential is similar. The basis for the approximating subspaces $\mathcal{U}^N$ were taken to be tensor products of the standard linear spline basis elements $\varphi^n_i$ corresponding to the uniform mesh $\{0,\frac{1}{n}, \frac{2}{n},...,\frac{n-1}{n},1\}$ on $[0,1]$, and the characteristic function basis $\chi^m_j$ for the interval $[a,b]$.  The $j^{th}$ element corresponds to the $j^{th}$ sub-interval $[a+(j-1)\frac{b-a}{m}),a+j\frac{b-a}{m})$, $j = 1,2,...,m$. In this way $\mathcal{U}^N=span\{\xi_{i,j}^N\}$, $i=1,2,...,n$, $j=1,2,...,m$ where $\xi_{i,j}^N(\eta,q)=\varphi_i^n(\eta)\chi_j^m(q)$, $\eta\in[0,1]$, $q\in[a,b]$ with $dim(\mathcal{U}^N)=nm$. Using standard estimates \cite{SCH} it is not difficult to show that Assumption (v) holds. 

Re-numbering $\xi_{i,j}^N$'s so that $\xi_{i,j}^N=\xi_k^N$ where $k=(i-1)n+j$ and letting $\Psi_k^N=[\psi_i^N]_{i=1}^{nm}\in\mathbb{R}^{nm}$, the matrix representation for the operators $\mathcal{A}^N$ are given by $[\mathcal{A}^N]=-(M^N)^{-1}K^N$ with
\begin{align*}
    M_{r,s}^N&=M_{r,s}^N(a,b,\theta)=\left<\xi_r^N,\xi_s^N\right>_{\mathcal{H}}\\
    &=\int_a^b\int_0^1\xi_r^N\xi_s^Nf(q;a,b,\theta)d\eta dq=\int_a^b\chi_j^m\chi_l^mf(q;a,b,\theta)dq\int_0^1\varphi_i^n\varphi_k^nd\eta,
\end{align*}
%
%
\begin{align*}
    K_{r,s}^N&=K_{r,s}^N(a,b,\theta)=\mathcal{a}(q;\xi_r^N,\xi_s^N)=\int_a^bq\int_0^1\frac{\partial\xi_r^N}{\partial\eta}\frac{\partial\eta_s^N}{\partial\eta}f(q;a,b,\theta)d\eta dq\\
    &=\int_a^bq\chi_j^m\chi_l^mf(q;a,b,\theta)dq\int_0^1\varphi_i^{n\prime} \varphi_k^{n\prime}d\eta,
\end{align*}
where $r=(j-1)n+i$, $s=(l-1)n+k$, $i,k=1,2,...,n$, $j,l=1,2,...,m$.

We also have
\begin{align*}
    B_r^N&=B_r^N(a,b,\theta)=\int_a^b\xi_r^N(1,q)f(q;a,b,\theta)dq=\varphi_i^n(1)\int_a^b\chi_j^mf(q;a,b,\theta)dq,
\end{align*}
%
%
\begin{align*}
    C_s^N(a,b,\theta)=\int_a^b\xi_s^N(1/3,q)f(q;a,b,\theta)dq-\varphi_k^n(1/3)\int_a^b\chi_l^m(q)f(q;a,b,\theta)dq,
\end{align*}
$r,s=1,2,...,nm$, $r=(j-1)n+i$, $s=(l-1)n+k$, $i,k=1,2,...,n$, $j,l=1,2,...,m$.

With the density $f=f_0(\cdot;\rho) = f_0(\cdot;(a,b,\theta))$ as given in Definition \eqref{trun.exp.fam} above, if we define
\begin{align*}
    f_1(\alpha,\beta;\rho)=\int_{\alpha}^{\beta}f(q;\rho)dq \ \ \ \ \text{and} \ \ \ \ f_2(\alpha,\beta;\rho)=\int_{\alpha}^{\beta}qf(q;\rho)dq,
\end{align*}
it is a straightforward, albeit somewhat tedious, exercise to compute the partial derivatives $\frac{\partial f_i}{\partial\alpha}$, $\frac{\partial f_i}{\partial\beta}$, $\frac{\partial f_i}{\partial\theta}$, $\frac{\partial f_i}{\partial a}$, $\frac{\partial f_i}{\partial b}$, $i=0,1,2$. These partial derivatives show up in the matrices that appear in the adjoint equations \eqref{eq6.1}-\eqref{eq6.5}. We tested our scheme on truncated uniform ($\rho=(a,b)$), exponential ($\rho=(R,\theta)$) and normal ($\rho=(a,b,\mu,\sigma)$) distributions. Our results are shown in Table \eqref{tab1} and Figure \eqref{fig1} below. In panels (a) - (c) of Figure \eqref{fig1}, we have plotted the converged estimated population models together with the data and the 75\% credible band for the truncated uniform, exponential and normal densities. The credible bands can be obtained directly from the solution to the population model.  Indeed, $\mathcal{q}$ is sampled using the estimated distribution and then $C(\mathcal{q})\mathcal{x}_{j}^N(\cdot,\mathcal{q})$ is evaluated at the sample q's where $\mathcal{x}_{j}^N$ is given by \eqref{eq5.5}. Now the $q$ dependence of the solution to the population model is only valid $\pi$ almost everywhere and our convergence framework is an $L_2$ (in $q$) theory.  Consequently, pointwise evaluation is, strictly speaking, undefined. However, the results appear to be useful so we have included them.   We are currently working on an extension of the results presented here that involves introducing parabolic regularization in $q$.  This will potentially allow us to justify pointwise evaluation in $q$ of the population model to obtain credible band.  It is interesting to note that the credible band for the exponential distribution is quite wide, almost to the point of making the population model not that useful. This is because the exponential distribution, especially one with a mean and variance of $\mu = 1/\theta = 3$, has a rather "fat" tail. Panels (d) and (f) of Figure \eqref{fig1} show the converging estimated pdfs for the truncated exponential and normal distributions, respectively.  Panel (e) shows how the output of the population model compares to the data when the resolution of the finite element discretizations of $q$ and $\eta$ and the truncation point of the densities are varied. It appears from the figure that it is the $q$ discretization that determines the rate of convergence, while a rather coarse $\eta$ discretization seems to suffice. We believe that this explains the slow convergence of $\theta$ (the exponential parameter) and $\sigma$ (the standard deviation of the normal) observed in Table \eqref{tab1} and panel (f) of Figure \eqref{fig1}. The truncation of the density appears to have only a negligible effect. We are currently investigating whether using smoother first order splines for the $q$ elements produces improved estimates and more rapid convergence.        
\begin{table}[ht]
\centering
\begin{tabular}{| c | c || c | c | c | c | c | c | c | c |}
\hline
\multicolumn{2}{|c||}{\boldmath$N$}  & \multicolumn{2}{c|}{\textbf{Uniform}} & \multicolumn{2}{c|}{\textbf{Exponential}} & \multicolumn{4}{c|}{\textbf{Normal}}                         \\ \hline
\boldmath$n$       & \boldmath$m$     & \boldmath$a^*$       & \boldmath$b^*$       & \boldmath$\theta^*$       & \boldmath$R^*$      & \boldmath$a^*$ & \boldmath$b^*$ & \boldmath$\mu^*$ & \boldmath$\sigma^*$ \\ \hline
4                & 4              & 1.76              & 4.27              & 2e-5                   & 3.61             & 2.61        & 5.44        & 4.05          & 0.62             \\
8                & 8              & 1.91              & 4.05              & 4e-5                   & 3.81             & 2.29        & 5.42        & 4.01          & 0.40             \\
16               & 16             & 1.94              & 4.00              & 0.20                   & 4.34             & 2.17        & 5.42        & 4.01          & 0.37             \\
32               & 32             & 1.95              & 3.99              & 0.30                   & 5.95             & 2.15        & 5.42        & 4.00          & 0.35             \\
64               & 64             & 1.96              & 3.99              & 0.30                   & 11.08            & 2.14        & 5.42        & 4.00          & 0.35                 \\ \hline
\multicolumn{2}{|c||}{True Values} & 2                 & 4                 & 1/3                    & ---      & ---                 & ---        & 4             & 0.25             \\ \hline
\end{tabular}
\caption{Convergence results for Examples 6.1, 6.2 and 6.3; estimation of the parameters in truncated uniform, exponential and normal distributions.}
\label{tab1}
\end{table}

\begin{figure}[ht]
\begin{center}
\includegraphics[scale=0.19]{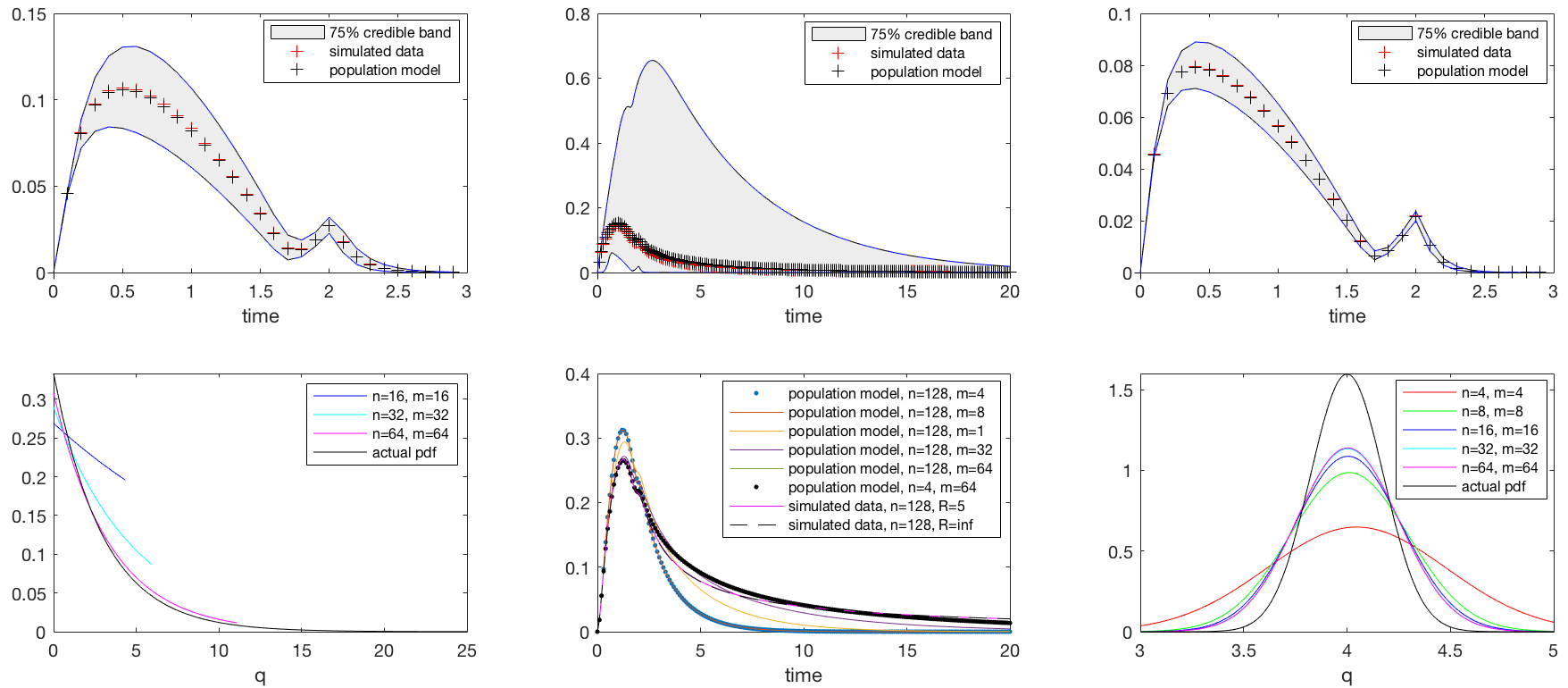}
\caption{Top row, starting from the left: Data, converged estimated population model and $75\%$ credible band for (a) Example 6.1 Truncated uniform distribution; (b) Example 6.2 Truncated exponential distribution; (c) Example 6.3 Truncated normal distribution. Bottom row, starting from the left: (d) Example 6.2 Converged pdfs for truncated exponential distribution; (e) Example 6.2 Data and Estimated population model for various values of $R$, $n$ and $m$; (f) Example 6.3 Converged pdfs for truncated normal distribution.}
\label{fig1}
\end{center}
\end{figure}

\subsection{Example 6.4; Two Random Parameters; Truncated Bi-variate Normal Distribution}\label{second.examp}

In this example we consider the system \eqref{eq6.5}-\eqref{eq6.11}, but instead of the Dirichlet boundary condition \eqref{eq6.2} at $\eta=0$, we take the Robin boundary condition \eqref{eq6.3} at $\eta=0$. In this case, $q=[q_1,q_2]$ is the vector of random parameters with $q\in D = Q=[a,b]\times[c,d]$, $H=L^2(0,1)$, $V=H^1(0,1)$, $W=H^2(0,1)$, and $Dom(A(q))=\{\varphi\in H^2(0,1):\Gamma_R\varphi=0, \Gamma_1\varphi=0\}$ and $\Gamma(q)=\Gamma_1$. The sesquilinear form on $V\times V$ is given by $a(q;\varphi,\psi)=q_1\int_0^1\varphi'\psi'd\eta+\varphi(0)\psi(0)$ with $<b(q),\psi>_{V^*,V}=q_2\psi(1)=q_2\delta(\cdot-1)$, $\psi\in V$, and $<c(q),\psi>_{V^*,V}=<c,\psi>_{V^*,V}=\psi(0)$, $\psi\in V$ where we have set $\eta_0=0$. In this case $N=(n,m_1,m_2)$, where $n$ is again the level of discretization of the space variable $\eta$ and $m_i$ is the level of discretization of $q_i$, $i=1,2$. Once again the approximating subspaces were constructed using tensor products, $\mathcal{U}^N=span\{\xi_{i,j,k}^N\}$, $i=0,1,2,...,n$, $j=1,2,...,m_1$, $k=1,2,...,m_2$ where $\xi_{i,j,k}^N(\eta,q_1,q_2)=\varphi_i^n(\eta)\chi_j^{m_1}(q_1)\chi_k^{m_2}(q_2)$, $\eta\in[0,1]$, $q_1\in[a,b]$, $q_2\in[c,d]$ with $dim(\mathcal{U}^N)=(n+1)m_1m_2$.

In this example the truncated exponential family was based on the bivariate normal. Once again, it is possible to compute all the partial derivatives (although of course their evaluation requires the numerical evaluation of single and double integrals) that are required to form the matrices that appear in the state and adjoint equations \eqref{eq6.1}-\eqref{eq6.5}. We obtained simulated data by generating samples for $\mathcal{q}$ from a $N(\bar{\mu},\bar{\Sigma})$ distribution with $\bar{\mu}=\begin{bmatrix}12\\ 10\end{bmatrix}$ and $\bar{\Sigma}=
\begin{bmatrix}
9 & 3\\
3 & 5
\end{bmatrix}$.
%
%
\begin{table}[H]
\centering
\begin{tabular}{ | c | c | c || c | c | c | c | c | c |}
 \hline
 \boldmath$n$ & \boldmath$m_1$ & \boldmath$m_2$ & \boldmath$a^*$ & \boldmath$b^*$ & \boldmath$c^*$ & \boldmath$d^*$ & \boldmath$\mu^*$ & \boldmath$\sigma^*$\\
 \hline
 4 & 8 & 8 & 5.88 & 18.15 & 4.85 & 14.63 & $\begin{bmatrix}11.72\\9.88\end{bmatrix}$ & $\begin{bmatrix}
12.13 & 5.76\\
5.76 & 7.35
\end{bmatrix}$\\
 8 & 8 & 8 & 5.67 & 18.35 & 5.17 & 14.46 & $\begin{bmatrix}11.68\\9.87\end{bmatrix}$ & $\begin{bmatrix}
10.15 & 4.04\\
4.04 & 5.97
\end{bmatrix}$\\
 16 & 8 & 8 & 5.79 & 18.17 & 5.06 & 14.66 & $\begin{bmatrix}11.67\\9.86\end{bmatrix}$ & $\begin{bmatrix}
9.29 & 3.03\\
3.03 & 5.21
\end{bmatrix}$\\
 \hline
\end{tabular}
\caption{Convergence results for Example 6.4; estimation of the parameters in truncated bivariate normal distribution.}
\label{tab2}
\end{table}
Our results are shown in Table \eqref{tab2} and Figure \eqref{fig2}, where it can be seen that we obtained reasonably good approximations to the actual parameters that we used to simulate the data. We parameterized the covariance matrix as $\Sigma=L^TL$, where the $2 \times 2$ matrix $L$ is upper triangular with $L_{11}$ and $L_{22}$ both positive so as to guarantee that at each step in the optimization, $\Sigma$ is positive definite symmetric. The plot of the optimal joint density in the left hand panel of Figure \eqref{fig2} correspond to $n=16$ and $m_1=m_2=8$. In the right hand panel of Figure \eqref{fig2} we have plotted the output of the fit population model and the $75\%$ credible band.  Once again, we believe that the rate of convergence could be improved by using linear splines rather than piece-wise constant elements to discretize the random parameters $q$. 
\begin{figure}[ht]
\begin{center}
\includegraphics[scale=0.24]{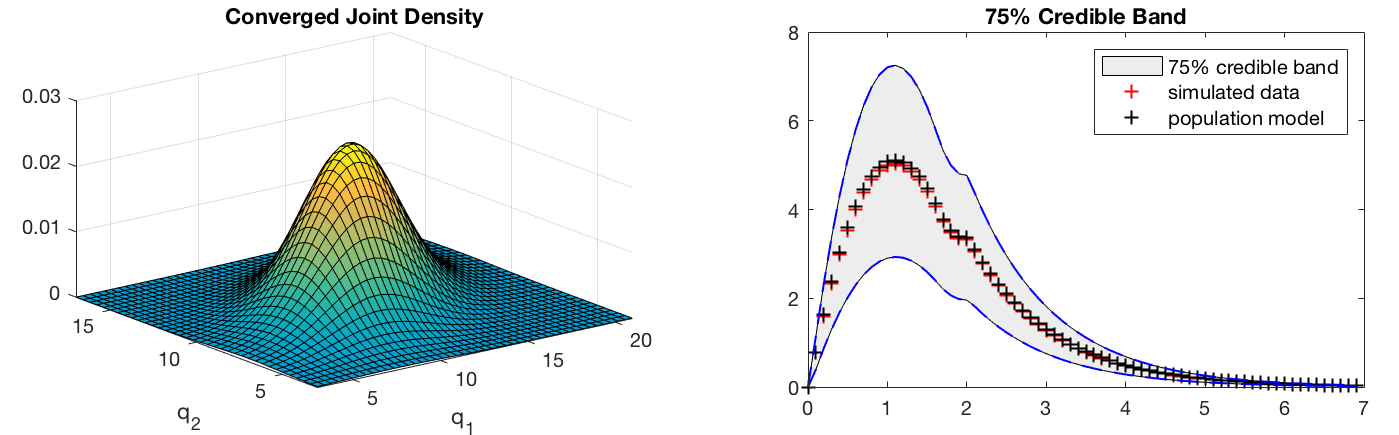}
\caption{Left hand panel: Example 6.4 Estimated bivariate normal joint density with $n=16$ and $m_1=m_2=8$; Right hand panel: Example 6.4 Data, estimated population model and $75\%$ credible band for truncated bivariate normal distribution.}
\label{fig2}
\end{center}
\end{figure}

\section{Concluding Remarks}\label{conc.rem}

We are currently working on a number of applications and extensions of the results presented here. Specifically, we are looking at applying our approach to actual experimental and clinical BrAC and TAC data collected in both the lab/clinic and the field using two different transdermal alcohol biosensors from a number of different individuals that include several drinking episodes occurring over a time period of several days. We are developing deconvolution schemes based on population models fit using the approach discussed here that, given an output signal, will provide a population based estimate for the input together with credible bands obtained directly from the deconvolved input signal and not requiring simulation. We are also looking at extensions of the ideas presented here to the solution of the LQR and LQG compensator problems wherein the infinite dimensional linear regularly dissipative dynamics and quadratic performance index involve random parameters.

In our treatment here, we assumed that the probability measures describing the distribution of the random parameters were defined in terms of parameterized families of joint density functions. We are looking at developing numerical schemes and an associated convergence theory for estimating the shape of the density directly. We also hope to be able to apply the convergence theory based on the Prohorov metric on a space of measures developed in \cite{BT} more directly to the class of problems that we have discussed here. More precisely, we would like to be able to eliminate the assumption that the measures are defined in terms of a density, and estimate the measure directly. We believe that such a theory may be possible by assuming that our approximating subspaces are required to satisfy additional regularity (i.e. smoothness) assumptions; in particular that they are required to be contained in the domain of the operator. Then by making use of a slightly different version of the Trotter-Kato semigroup approximation theorem (see, for example, \cite{BBC}) we believe it may now be possible to verify the hypotheses of the more general convergence theorem established in \cite{BT} for the estimation of the probability measures directly, rather than by estimating an associated density.

\section*{References}
\bibliography{main}{}
\bibliographystyle{plainurl}

\end{document}